\documentclass[11pt]{amsart}
\usepackage{hyperref, upref, comment}
\usepackage{amsmath,amsfonts,amscd,amssymb,amsthm}
\usepackage{longtable}
\usepackage[english]{babel}
\usepackage[utf8]{inputenc} 
\usepackage[active]{srcltx}
\usepackage[T1]{fontenc}
\usepackage{graphicx}
\usepackage{pstricks}
\usepackage{bbm}
\usepackage{MnSymbol}
\usepackage{stmaryrd}
\usepackage{nicefrac}
\usepackage{calrsfs}
\usepackage{mathtools}
\usepackage{mdwlist}   
\usepackage{epstopdf} 

\newtheorem{theorem}{Theorem}[section]

\newtheorem*{claim*}{Claim}

\newtheorem{coro}[theorem]{Corollary}
\newtheorem{lem}[theorem]{Lemma}
\newtheorem{pro}[theorem]{Proposition}

\newcommand{\be}[1]{\begin{equation}\label{#1}}
\newcommand{\ee}{\end{equation}}
\numberwithin{equation}{section}
\newcommand{\eps}{\varepsilon}
\newcommand{\ttau}{\tilde{\tau}_s}
\newcommand{\teps}{\tilde{\varepsilon}}
\newcommand{\tdelta}{\tilde{\delta}}
\newcommand{\tmu}{\tilde{\mu}_s}
\newcommand{\NI}{|\tilde{N}_I|}
\newcommand{\ba}[1]{\begin{align}\label{#1}}
\newcommand{\ea}{\end{align}}
\numberwithin{equation}{section}
\newcommand{\ben}{\begin{equation*}}
\newcommand{\een}{\end{equation*}}
\numberwithin{equation}{section}

\begin{document}
\title[Localization in random geometric graphs]{Localization in random geometric graphs with too many edges}
\author{Sourav Chatterjee}
\address{\newline Department of Statistics \newline Stanford  University\newline Email: \textup{\tt souravc@stanford.edu}}
\thanks{Research partially supported by ERC AG ``COMPASP'' and NSF grant DMS-1005312}
\author{Matan Harel}
\address{\newline Department of Mathematics \newline Tel Aviv University\newline Email: \textup{\tt mataharel8@tauex.tau.ac.il}}


\begin{abstract}
We consider a random geometric graph $G(\chi_n, r_n)$, given by  connecting two vertices of a Poisson point process $\chi_n$ of intensity $n$ on the $d$-dimensional unit torus whenever their distance is smaller than the parameter $r_n$. The model is conditioned on the rare event that the number of edges observed, $|E|$, is greater than $(1 + \delta)\mathbb{E}(|E|)$, for some fixed $\delta >0$. This article proves that upon conditioning, with high probability there exists a ball of diameter $r_n$ which contains a clique of at least $\sqrt{2 \delta \mathbb{E}(|E|)}(1 - \varepsilon)$ vertices, for any given $\varepsilon >0$. Intuitively, this region contains all the ``excess'' edges the graph is forced to contain by the conditioning event, up to lower order corrections. As a consequence of this result, we prove a large deviations principle for the upper tail of the edge count of the random geometric graph. The rate function of this large deviation principle turns out to be non-convex.
\end{abstract}

\maketitle


\section{Introduction}\label{Intro}
The random geometric graph is a simple stochastic model, first studied in \cite{Hafner} in 1972, for generating a graph: given the parameters $n$ and $r$, consider a Poisson point process of intensity $n$ on the $d$-dimensional unit torus, equipped with a translation-invariant metric inherited from a norm $\| \cdot \|$ on Euclidean space (which may not be the Euclidean norm), and declare an edge between any two vertices that are at distance $\le r$ from each other.

Unlike the well-known Erd\H{o}s--R\'enyi random graph, the random geometric graph's definition leads to strong dependence between edges: if three vertices form a ``V'' shaped graph, they are far more likely to have the third edge of the triangle than if no assumption were made on the other edges, as a consequence of the triangle inequality.

Many properties of this graph model have been studied. The classic monograph of Mathew Penrose \cite{Penrose2} studies results pertaining to many graph-theoretical functions of random geometric graphs, including (but not limited to) laws of large numbers and central limit theorems for subgraph counts, independence number, and chromatic number, as well as many properties connected to the giant component. Many of the results presented in this monograph have been improved and generalized by Penrose and others in the years since its initial publication. Besides this, there have been investigations into other probabilistic features, such as threshold functions for cover times and mixing times \cite{Avin} and thresholds for monotone graph functions \cite{Goel}. This list is far from comprehensive, of course, and the random geometric graph remains an active object of research.

The random geometric graph is also closely related to the random connection continuum percolation model. In that model, the vertex set is given by an (almost surely infinite) Poisson point process of fixed intensity on $\mathbb{R}^d$, and two points are connected with some probability that varies (and usually decreases) with their distance. In particular, the special case in which the radius of connection is deterministically fixed at 1 was the model that initiated the study of this kind of random geometry, in the seminal paper of Gilbert \cite{Gilbert}. The properties of interest in this model are the existence of an infinite connected component, as well as the behavior of the subset of $\mathbb{R}^d$ that is at distance at most 1 from one of the vertices of the graph (the so-called ``Poisson blob'') and its complement (the ``vacant set''). Continuum percolation is treated in detail in a book-length monograph by Meester and Roy \cite{Meester}, as well as in the book by Grimmett~\cite{Grimmett}.

Most of the work done on random geometric graphs is concerned with either the behavior of a typical graph --- the graph we are likely to see for a given $r$ as $n$ goes to infinity --- or typical deviations from that behavior --- that is, central limit theorems. In this paper, we are concerned with the behavior of the model conditioned on a rare event. Specifically, we will study the random geometric graph conditioned on having many more edges than is expected (a formal description will follow). The large deviation regime of the upper tail of any subgraph count of the random geometric graph is not well understood, though some bounds are available: Janson \cite{Janson} established concentration inequalities for $U$-statistics, a general class of statistics which includes the subgraph counts we are interested in. These upper bounds work in very general settings, but are not tight, even up to constants in the exponent. Large deviation principles have been proven for functionals of random point processes in which the contribution of any particular vertex is uniformly bounded \cite{Schreiber}, but no such bound is known for functionals with possibly large influence, such as the edge count of the graph.

As motivation for this detailed study, we consider the problem in a more familiar context: the ``infamous upper tail'' \cite{Janson2} of the triangle count $T$ in the Erd\H{o}s--R\'enyi random graph, $G(n,p)$. After many years of development of increasingly strong bounds, the first breakthrough was made by Kim and Vu \cite{Vu} and Janson et al.~\cite{Janson3} independently, who proved that, for any $\delta >0$ and whenever $p \gg (\log n)/n$,
\[
\exp[ - c(\delta) n^2 p^2 \log(1/p)] \leq \mathbb{P}[T > (1 + \delta)\mathbb{E}[T]] \leq \exp[- C(\delta) n^2 p^2] \, ,
\]
where $c(\cdot)$ and $C(\cdot)$ depend only on $\delta$. Recently, there has been renewed interest in these type of tail estimates. In 2010, Chatterjee \cite{Chatterjee2} and Demarco and Kahn \cite{Demarco} (in independent works) established the correct order of the upper tail of triangles and other small cliques by adding the missing logarithmic term to the upper bound, without providing good control of the leading-order constants.  The work of Chatterjee and Dembo~\cite{Chatterjee3}  on nonlinear large deviations proved that the upper tail probability can be described in terms of a continuous variational problem when $p$ is vanishing sufficiently slowly --- namely, when $n^{-1/42} \ll p \ll 1$. Generalizations and expansions of the approach by Eldan \cite{Eldan} established the variational equivalence for $n^{-1/18} \log n \ll p \ll 1$; Cook and Dembo \cite{Cook} proved the result for $n^{-1/3} \ll p \ll 1$, and Augeri \cite{Augeri} did the same for $n^{-1/2} (\log n)^2 \ll p \ll 1$. Lubetzky and Zhao \cite{Lubetzky} solved the variational problem for triangles (Bhattacharya, Ganguly, Lubetzky and Zhao \cite{bglz} did the same for more general subgraphs), which thus calculated both the order and the leading-order constant for the upper tail question in a certain regime of sparse Erd\H{o}s--R\'enyi random graphs. The main results and ideas from this body of work are summarized in the survey article~\cite{chatterjeesurvey}. Recently, Harel, Mousset, and Samotij \cite{HMS} used a combinatorial approach to prove that the upper tail probability of the subgraph count of any fixed, regular graph can be expressed in terms of the solution of a discrete variational problem for nearly all values of $p$ where localization is believed to hold. Unfortunately, all the papers described above are only valid for functions of independent Bernoulli random variables, and are therefore not applicable to the problem we are studying here.

In this work, we use the properties the random geometric graph inherits from the geometry of $\mathbb{R}^d$ to evaluate the upper tail large deviation rate function. In addition, we provide a ``structure theorem'' to describe the graph-theoretical structure of the model conditioned on having too many edges. Specifically, we show that such a conditional model exhibits {\em localization}. Heuristically, this phenomenon can be described as a scenario in which a small number of vertices will contribute almost all the extra edges that we require the graph to exhibit, while the edge count of the ``bulk'' of the graph will remain largely unchanged, in some weak sense. Furthermore, we will show that the geometry of the localized region has the shape of a ball in the given norm (we will make these two statements more precise at the end of the next section). Outside of the aforementioned works of Lubetzky and Zhao \cite{Lubetzky}, Bhattacharya et al.~\cite{bglz}, and Harel, Mousset and Samotij ~\cite{HMS} in the Erd\H{o}s--R\'enyi model, this work is the only (as far as the authors are aware) to establish that the large deviation regime of a subgraph count is (weakly) equivalent to planting a combinatorial structure in the usual, unconditional graph.

The fact that large deviation events may be dominated by configurations with a small number of very large contributions was known relatively early in the history of large deviation theory: a survey by Nagaev \cite{Nagaev}, summarizing a series of papers written in the Soviet Union in the 1960's and 70's, includes this observation for sums of i.i.d. random variables with stretched-exponential tails. In our context, the natural combinatorial structure for creating many edges with a small number of edges is a ``giant clique''. The clique number, the (typical) size of the largest clique of the random geometric graph, falls under the general class of scan statistics, and has been shown to focus on two values with high probability for certain values of $r$ (see \cite{Penrose4}, \cite{Muller}); however, these works do not explore the large deviation regime. Our work uses techniques from large deviations, concentration inequalities, convex analysis, and geometric measure theory. A key component in the proof is a technique for proving localization that has previously appeared in \cite{Talagrand1} and \cite{Chatterjee}.
\section{Definitions and Main Results}\label{MainResultsSection}
Let $\chi_n$ be a Poisson Point Process of intensity $n$ on the $d$-dimensional unit torus $\mathbb{T}^d = [0,1]^d$. For any $S \subset \mathbb{T}^d$, we denote the restriction of $\chi_n$ to $S$ by $\chi_n(S)$. Let $N := |\chi_n|$. Recall that $N$ is a Poisson random variable with mean $n$, and conditional on $N$, $\chi_n$ is just a set of $N$ points, each chosen independently and uniformly at random. Let $r_n$ be a positive sequence that  decreases to $0$ as $n\to \infty$, and $\| \cdot \|$ be some norm on $\mathbb{R}^d$ that induces a translation-invariant metric on $\mathbb{T}^d$. We define the random geometric graph $G(\chi_n,r_n) : = (V,E)$, where $V = \chi_n = \{v_1,\ldots, v_N\}$, enumerated arbitrarily, and $E$ is the set of unordered pairs $\{i,j\}$ such that $\| v_i - v_j\| \leq r_n$. Figure 1 shows a particular instance of $G(\chi_{150},0.1)$.

Letting $1_{i,j}$ be the indicator that there is an edge between $v_i$ and $v_j$, we can calculate the expected value of $|E|$, the number of edges in the
graph:
\begin{align*}
\mathbb{E}(|E|) &= \mathbb{E}\left(\displaystyle \sum_{1\le i<j\le N} 1_{i,j} \right) = \mathbb{E}\left[ {N \choose 2} \mathbb{E}(1_{1,2} \mid N) \right]\\
&= \frac{n^2}{2} \cdot \mathbb{P}(\|v_1-v_2\|\le r_n) \, .
\end{align*}
Denoting Lebesgue measure on both $\mathbb{R}^d$ and $\mathbb{T}^d$ by $\lambda(\cdot)$, we define
\[
\nu := \lambda\left[\left\{x \in \mathbb{R}^d: \| x\| \leq 1 \right\}\right] \,
\]
to be the volume of the unit ball in the norm $\| \cdot \|$. Then,by translation invariance of the metric induced on $\mathbb{T}^d$,  $\mathbb{P}(\|v_1-v_2\|\le r_n)$ is simply $\nu r_n^d$, as long as $r_n$ is sufficiently small (to ensure the ball on the torus has the same measure as the one in $\mathbb{R}^d$)
 \begin{equation*}
\mu_n := \mathbb{E}(|E|) = \frac{\nu \cdot n^2 r_n^d}{2} \, .
\end{equation*}
We can also compute the variance of $|E|$:
\begin{align}\label{eq:VarianceE}
&\text{Var}(|E|) = \mathbb{E}\left[ \text{Var}(|E| \mid N) \right] + \text{Var}\left( \mathbb{E}[|E| \mid  N] \right) \\ & = \mathbb{E}\left[ \mathbb{E}\left(\sum_{1 \leq i < j \leq N, 1 \leq i' < j' \leq n } (1_{i,j} - \nu r_n^d) (1_{i',j'} - \nu r_n^d) \mid N \right)\right] \nonumber \\
&\qquad \qquad + (\nu r_n^d)^2 \text{Var}\left[{N \choose 2} \right] \nonumber\\ & = \frac{n^2}{2} \left(\nu r_n^d - \nu^2 r_n^{2d}\right) + \left(n^3 + \frac{n^2}{2}\right)\nu^2 r_n^{2d} \nonumber\\ & = \mu_n \left(1 + 2 \nu n r_n^d \right) \nonumber,  
\end{align}
where we note that $(i,j) \neq (i',j')$ implies that the indicators $1_{i,j}$ and $1_{i',j'}$ are conditionally independent. This implies that, as long as $\mu_n \rightarrow \infty$, $\text{Var}(|E|) \ll \mu_n^2$, and $|E|$ concentrates around its mean by Chebyshev's inequality.

For the rest of the article, we suppose the existence of a fixed constant $\delta^* >0$ such that, for all sufficiently large $n$,
\begin{equation}\label{rcond}
n^{ (\delta^*-2)/d } \leq r_n \leq n^{- \delta^*/d} \, .
\end{equation}
The lower bound ensures that the expected number of edges grows as a positive power of $n$; the upper bound excludes the possibility of $r_n = n^{-o(1)}$ --- that is, bounded above and below by $n^{\varepsilon}$ and $n^{-\varepsilon}$, respectively, for any fixed $\varepsilon >0$ and $n \geq n_0$, for some $n_0$ depending on $\eps$.  We will reuse the notation $n^{o(1)}$ throughout the paper in this sense, and we will allow the (implicit) $\varepsilon$ to depend on any fixed  parameter other than $n$. We define the parameter $p$ as
\begin{equation}\label{pdef}
p : = \displaystyle \lim_{n \rightarrow \infty} \frac{ \log \mu_n}{\log n}\, ,
\end{equation}
implicitly assuming that the limit exists. This ensures that $\mu_n = f(n) n^p$, where $f(n) = n^{o(1)}$. Notice that
\begin{equation}\label{pbounds}
\delta^* \leq p \leq 2 - \delta^*\, ,
\end{equation}
thanks to \eqref{rcond}. We will say the random geometric graph is {\em admissible} if $r_n$ satisfies \eqref{rcond} and the limit above exists.

\begin{figure}[t]
 \centering
\includegraphics[scale=0.45]{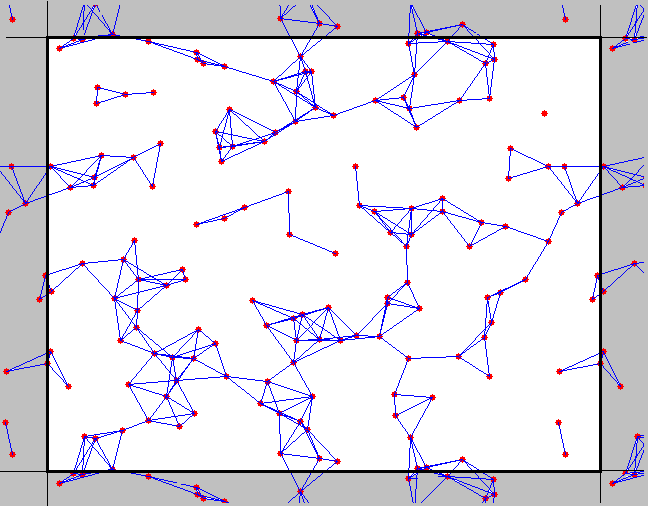}
\begin{center}
\caption{An instance of the random geometric graph $G(\chi_{150},0.1)$, with respect to the Euclidean norm. The graph has 148 vertices and 343 edges. The
gray area is the white unit square translated, to show periodicity}
\end{center}
\end{figure}

The following theorem is the main result of the paper:
\begin{theorem}\label{thm1}
Let $G(\chi_n,r_n)$ be an admissible random geometric graph model on $\mathbb{T}^d$ with respect to some norm $\| \cdot \|$. Define
\[
\tau_n := \nu \cdot (r_n/2)^d \, ,
\]
that is, $\tau_n$ is the volume of a ball of diameter $r_n$. Fix $\delta > 0$ and $\varepsilon >0$, and let $\mathcal{F}_n(\eps)$ be the event that there exists a ball $B$ of diameter $r_n$ such that
\begin{enumerate}
\item any convex set $S \subset B$ satisfies 
\begin{equation*}
 \left|\frac{|\chi_n(S)|}{\sqrt{2 \delta \mu_n}} - \frac{\lambda(S)}{\tau_n}\right| < \varepsilon\, ,
\end{equation*}
\item for any convex set $S'\subset B^c$ such that $\text{{\em diam}} (S') \leq r_n$ and $\lambda(S') > \varepsilon \tau_n$,
\begin{equation*}
\frac{|\chi_n(S')|}{\sqrt{2 \delta \mu_n}} < \varepsilon \cdot \frac{\lambda(S')}{\tau_n}\, .
\end{equation*}
\end{enumerate} 
Then
\[
\displaystyle \lim_{n\rightarrow \infty} \mathbb{P}\left[\mathcal{F}_n(\eps) \, \big \mid \, |E| > (1 + \delta) \mu_n\right] = 1 \, . 
\]

\end{theorem}

As a consequence of Theorem \ref{thm1}, we prove that the upper tail of the edge count of random geometric graphs satisfies a large deviation principle.
Recall that a sequence of non-negative random variables $X_n$ satisfies an {\em upper tail} large deviation principle with speed $s(n)$ and rate function $I(x)$ if, for any closed set
$F \subset (0,\infty)$,
\[
\displaystyle \limsup_{n\to \infty}\frac{1}{s(n)} \log \mathbb{P}\left(\frac{X_n - \mathbb{E}[X_n]}{\mathbb{E}[X_n]} \in F \right) \leq -\displaystyle \inf_{x \in F} I(x)\,
,
\]
and for any open set $G \subset [0,\infty)$,
\[
\displaystyle \liminf_{n\to\infty} \frac{1}{s(n)} \log \mathbb{P}\left(\frac{X_n - \mathbb{E}[X_n]}{\mathbb{E}[X_n]} \in G \right) \geq - \displaystyle \inf_{x \in G} I(x)\,
.
\]
(For more on large deviation principles and their applications, see e.g.~\cite{Dembo}.) The following theorem gives the upper tail large deviation principle for the number of edges in a random geometric graph.
\begin{theorem}\label{thmLDP}
Let $G(\chi_n,r_n)$ be an admissible random geometric graph model on the $d$-dimensional torus, with the same assumptions as in Theorem \ref{thm1}. Define
\[
I(x) := \left(\frac{2-p}{2}\right) \sqrt{2 x}\, ,
\]
where $p$ is defined as in \eqref{pdef}. Then $|E|$ satisfies an upper tail large deviation principle with speed $s(n) = \sqrt{\mu_n} \log n$ and rate function~$I(x)$.
\end{theorem}

There are several important features to the two main theorems of this paper: first, both describe models in which the number of edges significantly {\em exceeds} its mean. The lower tail of the edge count --- i.e.~events of the form $\{ |E| < (1 - \delta) \mu_n\}$ --- is likely to satisfy Poisson-like statistics. Its large deviation principle is expected to hold with speed $\mu_n$, and no special combinatorial structure like the ``giant clique'' of Theorem \ref{thm1} should appear.

Before we go on, let us comment on the precise properties of the giant clique given by our two main theorems. Since the rate function of Theorem~\ref{thmLDP} is strictly increasing, we know that, conditional on $\{|E| > (1 + \delta) \mu_n\}$, the event
\[
\left\{(1 + \delta') \mu_n > |E| > (1 + \delta) \mu_n\right\}
\]
occurs with high probability (i.e.~probability at least $1- \varepsilon$) for any $\delta' > \delta$ and $n$ sufficiently large. Now, if we set $S = B$ in the first stipulation of Theorem~\ref{thm1}, we see that the ball $B$ of diameter $r_n$ makes up a clique of at least $\sqrt{2 \delta \mu_n}(1 - \varepsilon)$ vertices --- and therefore at least $\delta \mu_n ( 1- 2 \varepsilon)$ edges. Since $\varepsilon$ and $\delta' - \delta$ are arbitrarily close to zero, we find that the clique in $B$ has $\delta \mu_n + o(\mu_n)$ edges, whereas the rest of the graph has $\mu_n + o(\mu_n)$ edges. This formalizes our earlier claim that `almost all extra edges in the conditional model are between points in $B$.'

Theorem \ref{thm1} also gives information about the internal geometry of the giant clique. If we pick $S$ to be a proper convex subset of the ball $B$, we find that $|\chi_n(S)|$ proportional to $\sqrt{2\delta \mu_n}$ times the density of $S$ in $B$ (again, up to lower order corrections). We restricted $S$ to be convex in order to preclude pathological sets, such as sets which are sparse but of large measure (e.g.~generalized Cantor sets) or have boundaries that take up a large amount of space. It should be possible to replace convexity with a weaker assumption. That being said, probing the Poisson point process $\chi_n$ with convex $S \subset B$ is enough to establish that the {\em conditional} process, restricted to $B$, is distributed roughly uniformly, up to errors that vanish in comparison to $\sqrt{\mu_n}$.

Finally, we would like to say that there are no other large cliques in $G(\chi_n, r_n)$ conditioned on $\{|E| > (1+ \delta) \mu_n\}$; unfortunately, Theorem \ref{thm1} does not provide this result. Instead, we can only be sure that every other clique outside of
the ``exceptional'' set $B$ has $o(\sqrt{\mu_n})$ vertices, that is, much smaller than the largest clique.

\section{The $s$-Graded Model}\label{SGradedSection}
Henceforth in the manuscript, we will suppress the subscript $n$ and write $\chi$, $\mu$, $\tau$ and $r$ instead of $\chi_n$, $\mu_n$, $\tau_n$ and $r_n$.

We now present an approximation of the random geometric model which allows us to replace the Poisson point process with a sequence of independent Poisson random variables. To do this, we first discretize space, and then produce a semi-metric on the resulting ``cells'' that approximates the norm $\| \cdot \|$ on the unit torus. We call this the $s$-graded model.

For a positive integer $s$, define
\[
m :=  \lfloor s/r \rfloor,
\]
so that
\begin{equation*}
\frac{s}{r} - 1 \leq m \leq \frac{s}{r}.
\end{equation*}
This definition and \eqref{pdef} imply that 
\begin{equation}\label{eq:mDef}
m^d = n^{2 - p + o(1)},
\end{equation}
where the constant in the $o(1)$ depends on $s$. Let $T = \{1,2, \dots, m\}^d$. Pick $I = (i_1, i_2, \dots, i_d) \in T$, and define
\[
A_I = \left[\frac{i_1-1}{m},\frac{i_1}{m}\right] \times \dots \times \left[\frac{i_d-1}{m},\frac{i_d}{m}\right].
\]
The $A_I$'s partition the unit torus into $m^d$ cubes (ignoring sets of measure $0$), each of volume $1/m^d$, and therefore, $X_I = |\chi(A_I)|$ is a Poisson random variable of mean
\begin{equation}\label{eq:Ddef}
\mathcal{D} := \frac{n}{m^d}.
\end{equation}
We now define a semi-metric on $T$, induced by the norm on torus:
\begin{equation}
\rho(I,J) = \inf_{x \in A_I^o,\,  y \in A_J^o}  \left\lceil m \|  x - y  \|\right\rceil
\end{equation}
where the circles indicate the interiors of the sets. Note that the $\rho(\cdot,\cdot)$ is always an integer. Moreover, $\rho(I,J) = z$ if $z$ is the smallest integer such that some point in $A_I^o$ and some point in $A_J^o$ are less than $z$ away, measured in units of $1/m$, the side length of the cubes. We force the points to be in the interior to prevent ``trivialities'', such as two adjacent cells being distance $0$, since they share a boundary. Note that $\rho(\cdot,\cdot)$ does not satisfy the triangle inequality, and hence is only a semi-metric. To see this, consider $\mathbb{T}^5$ under the Euclidean norm, and the cells $A_1 = A_{(1,\dots, 1)}$, $A_2 = A_{(2,\dots,2)}$ and $A_3 = A_{(3,\dots,3)}$. Since $A_1$ and $A_2$ share a corner, $\rho(A_1, A_2) = 1$, and the same holds for $\rho(A_2,A_3)$. However, $\rho(A_1, A_3) = \sqrt{5} > 1 + 1 = \rho(A_1,A_2) + \rho(A_2,A_3)$. But $\rho$ does satisfy a modified triangle inequality of the form $\rho(I,J) \leq \rho(I,K) + \rho(K,J) + C_d$, where $C_d$ depends only on the dimension and choice of norm, though we never make explicit use of this fact.

We are now ready to define the $s$-graded random geometric graph. Let $G_s(\chi,r) = (V,E_s)$ have the same vertex set as the original graph. For each vertex $v$, let $I_v$ be the index in $T$ such that $v \in A_{I_v}$; there is ambiguity on the boundary of the $A_I$'s, but that set has Lebesgue measure 0, and therefore it has no vertices of $\chi$, almost surely. We say $(v,w) \in E_s$ whenever $\rho(I_v, I_w) \leq s$. Heuristically, the $s$-graded model allows every point to wander inside a cubical ``cage'' of side-length $1/m$, and connects any two points that might be connected after we allow this mobility. In this framework, it is clear that $E_s$ becomes smaller as $s$ decreases. In fact, for sufficiently large $s$, $E_s$ is identical to $E$; unfortunately, this $s$ will be random. In formulating Theorem \ref{thm2}, the main theorem of this section (which is proved in Section \ref{sec:ProbSec}), we will let $s$ be an arbitrary positive integer, and discuss its asymptotic properties as $n$ goes to infinity. Later, in Sections \ref{sec:Geometry} and \ref{LDPSection}, we will take $s$ to sufficiently large, and show that the resulting approximation is good enough for our purposes.

Having defined the $s$-graded model, we now need to compute several quantities related to it, as we did for the random geometric graph in Section \ref{MainResultsSection}. We will denote $s$-graded model variables with tildes, to distinguish them from similar variables defined by the continuous geometry of the $\mathbb{T}^d$. We will say $G_s(\chi,r) = (V,E_s)$ is admissible whenever the random geometric graph $G(\chi,r)$ is admissible.

The major benefit of the $s$-graded model is that its edge count is very simple to express in terms of $X_I$, the number of points in each $A_I$:
\begin{align}
|E_s| &= \sum_{I \in T} \, \left[{X_I \choose 2}  + \frac{1}{2} \sum_{J\, : \, 0 < \rho(I,J) \leq s} X_I X_J \right] \label{eq:DiscreteDef}\\
&=  \frac{1}{2}\sum_{I \in T} \,  X_I \left[\biggl(\displaystyle \sum_{J\, : \, \rho(I,J) \leq s} X_J\biggr) -1\right].\notag
\end{align}
This random variable is defined in terms of i.i.d.\ random variables, which eases the analysis greatly. The geometric relations that define the edge count are now completely encoded by $\rho$. Finally, each $X_I$ only appears in finitely many terms in this expression (i.e. the number of terms involving $X_I$ is uniformly bounded in $n$). The ``finite range'' nature of the representation will play a major role in the proof presented.

We quantify this fact as follows: for any $I \in T$, let 
\[
\tilde{N}_I := \{J: \rho(I,J) \leq s\}.
\]
Thanks to translation invariance of $\rho$, the cardinality of this set is independent of the choice of $I$. Using this parameter, we can compute the expected number of edges in the $s$-graded random geometric graph easily: 
\begin{align}
\tmu  &:= \mathbb{E}(|E_s|)\label{musdef}\\
&= \displaystyle \sum_{I \in T} \mathbb{E}\biggl[{X_I \choose 2} + \frac{1}{2}  \sum_{J\, : \, 0 < \rho (I,J) \leq s} \mathbb{E}(X_I)
\mathbb{E}(X_J)\biggr]\nonumber \\
&= \frac{\NI m^d \mathcal{D}^2}{2} = \frac{\NI n^2}{2m^d}\, , \nonumber
\end{align}
where we recall that $\mathcal{D}$ is the mean of $X_I$, and the defining relation \eqref{eq:Ddef}. The variance of $|E_s|$ is also straightforward to calculate from the above representation, though the exact formula is messy. Instead, we produce an upper bound: the variance of $|E_s|$ can be thought of as the sum of $(\mathbb{E}[Q_I \cdot Q_{I'}]- \mathbb{E}[Q_I]^2)$, where $Q_I$ is the summand in \eqref{eq:DiscreteDef} and $I,I' \in T$. This quantity is maximized when $I = I'$, and is zero if the two terms are independent. Thus, we find that 
\begin{align*}
\text{Var}[|E_s|] \leq \sum_{I \in T} |\tilde{S}_I| \cdot \mathbb{E}\left( \biggl[{X_I \choose 2} + \frac{1}{2}  \sum_{J\, : \, 0 < \rho (I,J) \leq s} X_I X_J - \frac{\NI \mathcal{D}^2}{2}\biggr]^2 \right),
\end{align*}
where
\[
\tilde{S}_I := \{ J : \tilde{N}_I \cap \tilde{N}_J \neq \emptyset \} .
\]
A straightforward (if elaborate) computation can show that this implies that
\begin{equation}\label{eq:VarianceEs}
\text{Var}[|E_s|] \leq 16 |\tilde{S}_I| \NI^2 m^d \cdot \max\{\mathcal{D}^3, \mathcal{D}^2\}. 
\end{equation}
In Lemma \ref{lmm3} below, we will show that both $\NI$ and $|\tilde{S}_I|$ are uniformly bounded in $n$. Together with \eqref{musdef} and \eqref{eq:Ddef}, this implies that, for any $s$, $\text{Var}[|E_s|] \ll \tmu^2$, and hence $|E_s|$ concentrates around its mean by Chebyshev's inequality.

As before, we are interested in conditioning the $s$-graded model on the event $\{|E_s| > (1 +\tdelta) \tilde{\mu}_s\}$. Following Theorem \ref{thm1}, we expect that such conditional measures will be concentrated on configurations with many points on sets of diameter $s$ and maximal cardinality. We call a set of indices a $\textit{maximal clique set}$ if it is a subset of $T$ with diameter $\leq s$
that achieves the maximal cardinality of all such sets. We define
\begin{equation}\label{tausdef}
\ttau  := \max \{ |\mathfrak{I}| : \mathfrak{I} \subset T, \text{diam}(\mathfrak{I}) \leq s \} \, ,
\end{equation}
i.e.~$\ttau$ is the cardinality of a maximal clique set. Clearly $\ttau$ is increasing in $s$, and 
\begin{equation}\label{eq:ttaubound}
\ttau \geq \tilde{\tau}_1 \geq 2^d,
\end{equation} 
as the diameter of the set $\{ I = (\eta_1, \dots, \eta_n) : \eta_i \in \{1,2\}\}$ under the semi-metric $\rho(\cdot,\cdot)$ is exactly $1$, as all the $A_I$'s share a corner. We will also need an approximate notion of this geometric object: we say a set is a $\teps$-almost maximal clique set if its diameter is bounded above by $s$, and its cardinality is at least $(1 - \teps) \ttau $.

We can now state the equivalent to Theorem \ref{thm1} for the $s$-graded model:
\begin{theorem}\label{thm2}
Let $s$ be a positive integer. Consider $G_s(\chi,r)$, an admissible $s$-graded random geometric graph. For any $\tdelta >0$, define the event  $\mathcal{L}_n(\tilde{\delta})$ by
\[
\mathcal{L}_n(\tdelta) := \{ |E_s| > (1 + \tdelta) \tmu \} \, .
\]
For any $\teps >0$,  let $\mathcal{G}_{n,\tdelta}(\teps)$ be the event there exists a pair of sets $\mathfrak{B}$ and $\mathfrak{C}$ in $T$ such that 
\begin{enumerate}
\item $\mathfrak{B}$ is a $\tilde{\varepsilon}$-almost maximal clique set,
\item for all $I\in \mathfrak{B}$,
\[
\biggl|\frac{\ttau X_I}{(2 \tdelta  \tmu )^{1/2}} - 1 \biggr| < \teps \, ,
\]
\item $\mathfrak{C}$ satisfies 
\[
|\mathfrak{C}| < \teps \cdot \ttau \quad \text{and} \quad \sum_{I \in \mathfrak{C}} X_I <  \teps  \cdot (2 \tdelta \tmu)^{1/2},
\]
and
\item for all $J \in (\mathfrak{B} \cup \mathfrak{C})^c$,
\[
\frac{\ttau X_J}{(2 \tdelta \tmu )^{1/2}} < \teps .
\]
\end{enumerate}
There is a universal constant $\teps_0>0$ such that the following is true. Take any $\teps \in (0,\teps_0)$, any positive integer $s$,  and any three numbers  $0<\tdelta_0\le \tdelta \leq \tilde{\Delta}_0$. Then there is an integer $n_0$ depending only on $s$, $\teps$, $\tdelta_0$ and $\tilde{\Delta}_0$, such that whenever $n\ge n_0$, 
\begin{align*}
&\mathbb{P}[ \mathcal{G}_{n,\tdelta}(\teps)^c \cap \mathcal{L}_n(\tdelta)] \\
&\le 3 \exp\left(-(2 \tdelta \tmu)^{1/2} \left[\log \left(\frac{(2 \tdelta \tmu)^{1/2}}{\ttau \cdot \mathcal{D}}\right) -1 + (\teps/10)^{10}/2  \right]  \right).
\end{align*}
\end{theorem}
Because of its technical nature, Theorem \ref{thm2} warrants a short explanation. It turns out that it is possible to show that, for some $\eta >0$ 
\[
\mathbb{P}[\mathcal{L}_n(\tilde{\delta})] \geq \exp\left(-(2 \tilde{\delta} \tilde{\mu}_s)^{1/2} \left[\log \left(\frac{(2 \tilde{\delta} \tilde{\mu}_s)^{1/2}}{\tilde{\tau}_s \cdot \mathcal{D}}\right) -1 \right] - C n^{p/2 - \eta}  \right).
\]
This bound comes from explicitly ``planting" a maximal clique set where every cell includes exactly $\lceil (2 \tdelta \tmu)^{1/2}/\ttau \rceil$ vertices; we will not prove this fact, but Lemma \ref{LowerBoundH} will show a very similar computation for the edge count of the random geometric graph. In Lemma \ref{lmm3} below, we will show that $\NI$ is uniformly bounded in $n$ for any $s >0$. Together with the fact that $\tdelta$ is uniformly bounded above and below in $n$, this implies that $(2 \tdelta \tmu)^{1/2} = n^{p/2 + o(1)} \gg n^{p/2 - \eta}$. Therefore, Theorem \ref{thm2} shows that, with high probability, the event $\mathcal{G}_{n,\tdelta}(\tilde{\varepsilon})$ occurs in the conditional $s$-graded model.

The event $\mathcal{G}_{n,\tdelta}(\teps)$ produces a set $\mathfrak{B}$, which is very close to a maximal clique set, in which each $X_I$ is very close to $\sqrt{2 \tdelta \tmu}/\ttau$ --- the value we would expect if we were to spread the $\sqrt{2 \tdelta \tmu}$ vertices required to make a ``giant'' clique evenly among the $\ttau$ elements of a maximal clique set. In addition, we allow for an ``exceptional" set $\mathfrak{C}$, where some $X_I$'s may be much larger than this average amount. However, this exceptional set is made up of few indices, and includes few vertices of $\chi$, when compared with $\sqrt{2 \tdelta \tmu}$. Outside of these two sets, every $X_J$ is at most $\teps \sqrt{2 \tdelta \tmu}/ \ttau$ --- a lower order quantity when compared to the bounds on $X_I, I \in \mathfrak{B}$.

When this event fires, the conditional $s$-graded model has a clique with approximately $\tilde{\delta} \tilde{\mu}_s$ edges. We also know that the vertices are distributed roughly uniformly. Finally, we get a  quantitative estimate on the probability that the edge count of the $s$-graded model exceeds its mean {\em without} the desired structure occurring. Note that the constants and 10th power of $\teps$ that appears in the quantitative bound are somewhat arbitrary --- we made no attempts to optimize them.

Suppose that $\tilde{\varepsilon} <(2\tilde{\tau}_s)^{-1}$. In this case $\mathcal{G}_{n,\tdelta}(\tilde{\varepsilon})$ would require $|\mathfrak{B}| \geq \tilde{\tau}_s -1/2$ and  $|\mathfrak{C}| \leq 1/2$ --- i.e.~$\mathfrak{C}$ is empty and $\mathfrak{B}$ is a true maximal clique set. Thus, Theorem \ref{thm2} can be used to show that the $s$-graded model conditioned on $\mathcal{L}_n(\tilde{\delta})$ will include a maximal clique set housing a clique of at least $\tilde{\delta} \tilde{\mu}_s(1 - o(1))$ edges. Unfortunately, the quantitative estimate on the probability of $\mathcal{G}_{n,\tdelta}(\tilde{\varepsilon})^c \cap \mathcal{L}_n(\tilde{\delta})$ in this case is not sufficiently good to deduce Theorem \ref{thm1}. This is the reason for the introduction of the $\tilde{\varepsilon}$-almost maximal clique sets, which allow us to deduce a stronger upper bound on the probability that $\mathcal{G}_{n,\tdelta}(\tilde{\varepsilon})$ does not occur --- at the price of dealing with more flexible geometric constructions.

\section{Outline of the Proof} \label{OutlineSection}
Before embarking on a proper proof, we sketch the main ideas required. We recall that $\delta^* >0$ is a fixed positive number and that $p$ is given by $\lim_{n \rightarrow \infty} \,  \log \mu/\log n$. We will define 
\[
a := \frac{\delta^*}{25}.
\]
Later, we will also pick two positive real numbers $\alpha, \beta$ as some quantities depending on $p$ and $a$. All of these quantities will be fixed throughout the paper. We further note that, for any admissible graph, $r \rightarrow 0$ as $n \rightarrow \infty$. For the remainder of the paper, we will take the statement ``$n$ is sufficiently large'' to imply that $r$ is sufficiently small.

We begin by carefully analyzing the $s$-graded model. We order the indices $I$ by the size of $X_I$, the point counts of the $A_I$'s. Explicitly, we pick a bijection from $T$ to $\{1, 2, \dots , m^d \}$ such that
\[
X_1 \geq X_2 \geq \dots \geq X_{m^d} \, . 
\]
For notational convenience, we set 
\[
q = (2 \tilde{\delta} \tilde{\mu}_s)^{1/2}, \ \ \ w = \tilde{\tau}_s \cdot \mathcal{D}.
\]
Picking $a$ as above, we set $M = \lceil \mathcal{D} \rceil \cdot n^{a}$, and let $\mathcal{T}_M$ be the greatest $I$ such that $X_{I} \geq M$.  We define
\[
\mathfrak{I} = \{1, 2, \dots, \mathcal{T}_M\} \, , \text{ ordered by size as above,}
\]
to be the set of indices whose associated point counts $X_I$ exceed their mean (corrected for integrality) by a fixed polynomial factor. Furthermore, define
\[
Y_I := X_I \left(\log (X_I /\mathcal{D}) -1\right) + \mathcal{D} \, ,
\]
and
\[
Q(\mathfrak{I}) := \frac{2}{q^2} \displaystyle \sum_{I \in \mathfrak{I}} \biggl[{X_I \choose 2} + \frac{1}{2} \displaystyle \sum_{\substack{J \in \tilde{N}_I \cap
\mathfrak{I}\\J\ne I}} X_I X_J\biggr] \, ,
\]
and 
\[
V(\mathfrak{I}) := \frac{1}{q} \displaystyle \sum_{I \in \mathfrak{I}} X_I.
\]
The first quantity is an appropriately chosen convex function of the $X_I$'s, while the second is a scaled version of the number of edges with both endpoints in the $A_I$'s associated with $\mathfrak{I}$, and the third controls the number of vertices in $\mathfrak{I}$.

Let $\xi >0$ be a fixed constant independent of $n$. Consider the event 
\[
\mathcal{H}_\xi = \left\{Q(\mathfrak{I}) \geq 1 - \frac{\xi}{\log n}\right\} \bigcap  \left\{\frac{1}{q} \displaystyle \sum_{I \in \mathfrak{I}} Y_I  \leq \log(q/w) -1 + \xi\right\}\, .
\]
The main probabilistic analysis of this paper occurs in two sections: the first uses large deviation estimates to control sums of i.i.d.~random variables, and the second employs concentration inequalities for more complicated functions. Together, this work allows us to show that, for sufficiently large values of $n$ and small values of $\xi$,
\[
\mathbb{P}[ \mathcal{H}^c_\xi \cap \mathcal{L}_n(\tilde{\delta})] \leq 3\exp\left(-q \left[\log \left(\frac{q}{w}\right) -1 + \xi/2  \right]  \right).
\]
It turns out that, if we set $\xi \leq (\teps/10)^{10}$, {\em any} $\mathfrak{I} \subset T$ that satisfies both the quadratic lower bound and the convex upper bound that define $\mathcal{H}_\xi$ (as well as a mild bound on $\mathcal{T}_M$ and $V(\mathfrak{I})$) contains in it a $ \teps$-almost maximal clique set $\mathfrak{B}$ and an exceptional set $\mathfrak{C}$ that satisfy the four stipulations of $\mathcal{G}_{n,\tdelta}(\tilde{\varepsilon})$! This nontrivial statement implies $\mathcal{G}_{n,\tdelta}(\tilde{\varepsilon})^c  \subset\mathcal{H}_\xi^c$, whenever $\xi \leq (\teps/10)^{10}$ -- and, in particular, $\{\mathcal{G}_{n,\tdelta}(\tilde{\varepsilon})^c \cap \mathcal{L}_n(\tdelta) \}\subset\{\mathcal{H}_\xi^c \cap \mathcal{L}_n(\tdelta)\}$. This proves Theorem \ref{thm2}.

The proof of the above implication is not straightforward, and we will deduce it in several steps. We emphasize that this is a completely deterministic property of configurations that satisfy a certain set of inequalities. The next two paragraphs sketch the argument used to prove this implication.

Set $\mathcal{T}_V$ to be 
\[
\mathcal{T}_V := \min\left\{ k : V(\{1 , \dots , k\}) > 1 - \frac{2\xi}{\log n} \right\} \, \text{ and } \, \mathfrak{T}:= \{1, \dots, \mathcal{T}_V\}.  
\]  
Careful use of minimality and Jensen's inequality proves that 
\[
V(\mathfrak{T}) \leq 1 + \phi(\mathcal{T}_V), \, \, \quad \, \, \, \, Q(\mathfrak{T}) \geq 1 - \psi(\mathcal{T}_V) \, ,
\]
where $\phi(\cdot)$ and $\psi(\cdot)$ are explicit functions, bounded above by $1/(\log n)^{1/2}$, that are non-increasing in their arguments. One of the difficulties we encounter is that we do not have good upper bounds on $\mathcal{T}_V$, and thus must have bounds that improve whenever the parameter grows.

We set $\mathcal{T}_P$ to be the greatest integer $I$ smaller than $\mathcal{T}_V$ that satisfies $X_{I} >  \xi q /\ttau$. Setting $\mathfrak{P} = \{1, 2, \dots , \mathcal{T}_P \}$, we now have a set of indices whose associated $X_I$'s are commensurate with $q$. We proceed to show that the diameter of $\mathfrak{P}$ cannot exceed $s$ without violating either the lower bound on $Q(\mathfrak{T})$ or the upper bound on $V(\mathfrak{T})$. Together with technical estimates that force $\mathcal{T}_P \geq \tilde{\tau}_s(1 - \xi^{1/3})$, we find that $\mathfrak{P}$ is an $\xi^{1/3}$-almost maximal clique set. Moreover, a quantitative version of Jensen's inequality allows us to break $\mathfrak{P}$ into $\mathfrak{B}$ and $\mathfrak{C}$, the required sets. Finally, we can show that $X_{\mathcal{T}_P + 1}$ vanishes sufficiently quickly to completes the proof of Theorem \ref{thm2}.

We then move on to proving that Theorem \ref{thm2} implies Theorem \ref{thm1}. To do so, we first show that we can approximate any convex subset $S$ of a ball of diameter $r$ from both the inside and the outside by a union of $A_I$'s using the tools of geometric measure theory. Next, we use the classical isodiametric inequality to show that the $A_I$'s associated with a $s^{-1/20}$-almost maximal clique set approximate a ball of diameter $r$, in the sense of the Hausdorff metric.

Next, we fix $\varepsilon >0$, and show that, for sufficiently large $s$ and $\tdelta \in [(1 - \varepsilon/16)\delta,\delta]$, the event $\mathcal{G}_{n,\tdelta}(s^{-1/20})$ will imply $\mathcal{F}_n(\varepsilon)$. We then apply Theorem \ref{thm2} with $\tdelta$ as above and $\teps = s^{-1/20}$ to get an upper bound on the probability of $\{\mathcal{F}_n(\varepsilon)^c \cap \mathcal{L}_n(\tdelta)\}$. Combining this bound with a good lower bound on the probability of $\{ |E| > (1 +\delta) \mu\}$ (to be derived directly from the Poisson point process) and a well-known correlation inequality gives Theorem \ref{thm1}.

The final section of the paper proves the large deviation principle of Theorem \ref{thmLDP}. We use the first stipulation of Theorem \ref{thm1} and the $s$-graded model to compute the upper bound.

\section{Analysis of the $s$-Graded Model}\label{sec:ProbSec}
In this section we analyze the $s$-graded model and prove Theorem \ref{thm2}. At the very beginning, let us now fix a positive integer $s$ and numbers $0< \tdelta_0\le \tdelta \leq \tilde{\Delta}_0$. We will figure out the universal constant $\teps_0$ later. Throughout, whenever we say ``$n$ sufficiently large'', we will mean ``$n\ge n_0$ for some $n_0$ that depends only on $s$, $\teps$, $\tdelta_0$, and $\tilde{\Delta}_0$''.

\subsection{Controlling the Natural Parameters of the $s$-Graded Model} 
The geometric properties of the $s$-graded model are not quite comparable to those of the random geometric graph; most obviously, the $s$-graded model has a discrete geometry induced by the semi-metric $\rho(\cdot, \cdot)$ on $T$. We begin with a very useful lemma, which tells us that the parameters of the $s$-graded model are close to their appropriate equivalents on $\mathbb{T}^d$. To do so, we define three operators: first, let $\mathfrak{U}$ send a set of indices to the union of their associated $A_I$'s --- that is, for any $\mathfrak{I} \subset T$,
\begin{equation}\label{UnionDef}
\mathfrak{U}(\mathfrak{I}) := \displaystyle \bigcup_{I \in \mathfrak{I}} A_I \, .
\end{equation}
In the other direction, we must be more careful. Let $K \subset \mathbb{T}^d$, we define $\mathfrak{R}(K)$ and $\mathfrak{O}(K)$ to be the maximal (resp. minimal) subsets of $T$ such that
\begin{equation}\label{InnerOuterHullDef}
\mathfrak{U}\left(\mathfrak{R}(K)\right) \subset K  \quad \text{and} \quad K\setminus K' \subset \mathfrak{U}\left(\mathfrak{O}(K)\right) \, ,
\end{equation}
where $K'$ is some subset of $K$ of Lebesgue measure $0$; this modification allows us to not deal with certain trivialities. We note that $\mathfrak{R}(K)$ may be empty, and $\mathfrak{O}(K)$ may be $T$, even when $K$ or $\mathbb{T} \setminus K$ are nonempty. Alternatively, we may define $\mathfrak{O}(K)$ by
\[
\mathfrak{O}(K) := \{I \in T : \lambda( K \cap \mathfrak{U}(I)) > 0 \}.
\]

We recall several definitions: $\mu = \mathbb{E}(|E|)$, $\tilde{\mu}_s = \mathbb{E}(|E_s|)$ and $\tau = \nu (r/2)^d$. We set $\NI$ to be the number of indices satisfying $\rho(I,J)\leq s$, and $\tilde{S}_I =  \{ J: \tilde{N}_I \cap \tilde{N}_J \neq \emptyset \}$. Finally, $\ttau$ is the cardinality of a maximal clique set (as defined in \eqref{tausdef}). 
\begin{lem}\label{lmm3}
We have $E\subset E_s$, and there exist constants $C$, $s_0$, and $n_0$ depending only on the dimension and the chosen norm of the torus, such that, if $s\ge s_0$ and $n \geq n_0$, then
\[
\mu \leq \tilde{\mu}_s \leq \mu \left(1 + \frac{C}{s}\right)
\]
and
\[
m^d \tau \leq \tilde{\tau}_s \leq m^d \tau \left(1 + \frac{C}{s}\right)
\]
Furthermore, $\NI$, $|\tilde{S}_I|$, and $\tilde{\tau}_s$ are uniformly bounded in $n$. 
\end{lem}
In this section, we will only use this lemma to establish that certain quantities are uniform in $n$; in Section \ref{sec:Geometry}, we will strongly use the fact that the estimates become tight as $s$ grows.

\begin{proof}
Pick an arbitrary $I$ and consider $\mathfrak{U}(\tilde{N}_I)$. By definition of $\rho(\cdot,\cdot)$ and $s$, this set includes a ball of radius $r$ around $\textit{any}$ point in $A_I$. Therefore, any pair $(v,w) \in E$ must also be in $E_s$, giving the first stipulation. Since this inclusion holds for any configuration of the underlying Poisson Point process, this also gives $\mu \leq \tilde{\mu}_s$.

Now, define $\varsigma$ to be the diameter of the unit cube under the norm $\| \cdot \|$ - that is,
\begin{equation}\label{DiameterDefinition}
\varsigma := \sup_{x,y \in [0,1]^d} \|x - y \|.  
\end{equation}
Fix $I$, and let $x$ and $y$ be two fixed points in $A_I$ and $\mathfrak{U}(\tilde{N}_I)$, respectively. Letting $J$ be an index for which $y \in A_J$, we pick arbitrary points $z$ and $w$ in $A_I$ and $A_J$, respectively. Then the triangle inequality for $\| \cdot \|$ implies that
\[
\| x - y\| \leq \|x - z\| + \|z - w\| + \|w - z\| \leq \|z - w\|\ + \frac{2 \varsigma}{m} \, ,
\]
where we bound the first and last terms by $\varsigma/m$ using scaling of the norm. Since $z$ and $w$ are arbitrary, we can take an infimum over all choices of $z$ and $w$ in $A_I^\circ$ and $A_J^\circ$, respectively, and conclude that
\[
\| x- y\| \leq \frac{s + 2 \varsigma}{m} \leq \frac{ (s + 2\varsigma)r}{s-r} \, ,
\]
where we use that $m \geq s/r - 1$, by definition. Therefore, $\mathfrak{U}(\tilde{N}_I)$ is contained in a ball of radius $r(1 + 3\varsigma/s)$ around any point in $A_I$, for sufficiently large value of $s$ and $n$ (recalling that $r$ is vanishing in $n$). Since each $A_I$ is of measure of $m^{-d}$, we deduce that
\[
\NI = m^d \lambda\left(\mathfrak{U}(\tilde{N}_I)\right) \leq \nu m^d r^d \left(1 + \frac{3 \varsigma}{s}\right)^d \leq \nu m^d r^d \left(1 + \frac{6 d \varsigma}{s} \right) \, ,
\]
where the final inequality follows because $(1 +x)^d \leq 1 + (2d)x$ for all sufficiently small $x$. Substituting this into the definition of $\tilde{\mu}_s$ produces the desired inequality on $\tilde{\mu}_s$. Repeating a similar analysis will show that the set $\mathfrak{U} \left(\{J : \tilde{N}_I \cap \tilde{N}_J \neq \emptyset\}\right)$ is a subset of some ball of radius $2r(1 + 3 \varsigma/s)$, and thus 
\[
|\tilde{S}_I| \leq \nu m^d (2r)^d \left(1 + \frac{6d \varsigma}{s}\right).
\]

Next, we wish to control $\tilde{\tau}_s$. For the lower bound, let $B \subset \mathbb{T}^d$ be an arbitrary ball (in $\| \cdot \|$) of diameter $r$. Consider $\mathfrak{O}(B)$. By minimality, $\lambda(A_I \cap B) >0$ for every $I \in \mathfrak{O}(B)$. Therefore, 
\[
\max_{I,J \in \mathfrak{O}(B)} \left[\inf_{x \in A_I^o, y \in A_J^o}  \| x - y\| \right] \leq r  \, , 
\]
which implies, by the definition of $\rho(\cdot, \cdot)$, that the diameter $\mathfrak{O}(B)$ is at most $s$. Meanwhile, by inclusion, and the fact that $\lambda(A_I) = 1/m^d$ for every $I$,
\[
|\mathfrak{O}(B)| > m^d \lambda(B) = m^d \tau \, ,
\]
completing the lower bound.

For the upper bound, pick any $\mathfrak{W} \subset T$ such that 
\[
\lambda \left(\mathfrak{U}(\mathfrak{W})\right) \geq \tau \left(1 + \frac{ C}{s} \right)
\]
Applying the isodiametric inequality for finite dimensional normed spaces \cite[p. 93]{Burago} and choosing $C$ and $s_0$ sufficiently large gives
\[
\text{diam}(\mathfrak{U}(\mathfrak{W})) \geq r \left(1 + \frac{ C}{s} \right)^{1/d} \geq r\left(1 +  \frac{4 \varsigma}{s-r} \right) \, . 
\]
This implies that the diameter of $\mathfrak{W}$ is at least $s+1$. Therefore, any set $\mathfrak{W}$ of diameter at most $s$ must satisfy $\lambda \left(\mathfrak{U}(\mathfrak{W})\right) < \tau (1 + c/S)$, and 
\[
|\mathfrak{W}| = m^d \cdot \lambda \left(\mathfrak{U}(\mathfrak{W})\right)  \leq m^d \tau \left(1 + \frac{ C}{s} \right) \, ,
\]
as required.

The uniform bounds on $\NI$, $|\tilde{S}_I|$, and $\tilde{\tau}_s$ follow from $m^d \leq s^d /r^d$ and the above formulae.
\end{proof}

An immediate corollary to this theorem is that, assuming \eqref{pdef}, $\tmu = n^{p + o(1)}$. 

\subsection{Large Deviation Estimates}
The probabilistic bounds we need in this work are divided into two parts. The first involves good control on the deviation of sums of i.i.d.\ random variables. Our main tools here will be Chernoff bounds, as well as exact lower bounds.

Recall from Section \ref{OutlineSection} that $q = (2 \tilde{\delta} \tilde{\mu}_s)^{1/2}$, $w = \tilde{\tau}_s \cdot \mathcal{D}$, and $\mathcal{L}_n(\tilde{\delta}) = \{|E_s| > (1 + \tilde{\delta}) \tilde{\mu}_s\}$. By our assumptions on $\tdelta$, \eqref{eq:Ddef}, \eqref{musdef}, and Lemma \ref{lmm3}, we have that $q = n^{p/2 +o(1)}$ and $w = n^{p-1 +o(1)}$. Since $p/2 > p-1$ for any admissible $s$-graded model, we can increase $n$ to ensure that $q > 3 w$. We will assume this inequality for the rest of the paper.

We begin by recalling some classical bounds on the Poisson distribution (for proof, see \cite[pg. 35]{Dembo}, for example):
\begin{lem}\label{ChernoffBounds}
Let $X_I$ be a Poisson random variable with mean $\mathcal{D}$. Then, for any $t > \mathcal{D}$, 
\[
\mathbb{P}[X_I \geq t] \leq \exp (- t [\log (t/\mathcal{D}) -1] - \mathcal{D}) \, ,
\]
and for any $t < \mathcal{D}$, 
\[
\mathbb{P}[X_I \leq t] \leq \exp (- t [\log (t/\mathcal{D}) -1] - \mathcal{D}) \, . 
\]
\end{lem}
These bounds, which are given by explicitly computing exponential moment generating functions, are tight up to polynomial factors, and will be very important in the nearly exact computations we do in the proceeding lemmas.

We now define a random ordering of $T$ according to the $X_I$. Specifically, we pick a bijection from $T$ to $\{1, 2, \dots, m^d\}$ such that
\[
X_1 \geq X_2 \geq \dots \geq X_{m^d} \, .
\]
This bijection is not unique, as each $X_I$ is integer-valued, and there may be many $I$'s whose associated $X_I$'s are equal. However, all the statements will be true independently of the particular choice of bijection. Next, fix $a = \delta^*/25$, and define $M$ by
\begin{equation}\label{eq:DefinitionofM}
M := \lceil \mathcal{D} \rceil \cdot n^a \, . 
\end{equation}
The number $M$ is defined so to be a threshold of density for $X_I$ --- if $X_I < M$, we say it is in the {\em bulk} of the graph. We expect that, even conditional on $\mathcal{L}_n(\tilde{\delta})$, most indices $I$ will be in the bulk. To formalize this, we let 
\[
\mathcal{T}_M := \max \{ I : X_I \geq M \} \, .
\]
The next proposition controls the tail of $\mathcal{T}_M$:
\begin{pro}\label{CardinalityBoundPro}
Let
\[
\alpha = \min\{1 - p/2 - a/2,\, p/2 - a/2\} \, . 
\]
Let $\mathcal{A}$ be the event $\{ \mathcal{T}_M \geq n^{\alpha} \}$. Then, for all sufficiently large $n$,
\[
\mathbb{P}[\mathcal{A}] \leq  \exp\left(-n^{ p/2 + a/3}\right) \, . 
\]
\end{pro}
The number $\alpha$ will remain fixed to the value above for the remainder of the paper. We note that $\alpha < 2 -p$ for any admissible $s$-graded model, and therefore $n^{\alpha} \ll m^d$ (using \eqref{eq:mDef}). Thus, we find that, with very high probability, the complement of the bulk takes up a vanishing proportion of $T$.
\begin{proof}
The event $\mathcal{A}$ implies the existence of some $\mathfrak{W} \subset T$ such that, for all $I \in \mathfrak{W}$, $X_I > M$, and $|\mathfrak{W}| > \lceil n^{\alpha} \rceil$. By the union bound,
\[
\mathbb{P}[\mathcal{A}] \leq { m^d \choose \lceil n^{\alpha} \rceil} \cdot \mathbb{P}[X_I > M]^{n^{\alpha}} \, . 
\]
Using the upper tail bound in Lemma~\ref{ChernoffBounds} and the brutal bound ${m \choose k} < m^k$, this implies that
\begin{align*}
\mathbb{P}[\mathcal{A}] & \leq m^{d (n^{\alpha} +1)} \cdot \exp( - n^{\alpha} M \left[\log (M/\mathcal{D}) -1\right]) \\ & \leq  \exp \left( d (n^{\alpha} +1) \log m- n^{\alpha + a } \lceil \mathcal{D} \rceil \right) \, . 
\end{align*}
Since $\log m$ is bounded above by $C \log n$ for some uniform constant $C$ (by Lemma \ref{lmm3}), we can increase $n$ to ensure that 
\[
\mathbb{P}[\mathcal{A}] \leq \exp\left( - \frac{n^{\alpha + a} \cdot \lceil \mathcal{D} \rceil}{2}\right) \, . 
\]
If $\mathcal{D} \leq 1$, then the ceiling function is 1 and $\alpha = p/2 - a/2$. Increasing $n$ until $n^{p/2 + a/2}/2 > n^{p/2 + a/3}$ completes this case. If $\mathcal{D} >1$, then we bound $\lceil \mathcal{D} \rceil$ by $\mathcal{D}$ itself. By definition, 
\[
\mathcal{D} n^{\alpha + a} = n^{ \min\{3p/2 -1 + a/2, p/2 + a/2\} + o(1)}.
\]
In the case $p \geq 1$, the exponent is always minimized by the second choice. This completes the proof.
\end{proof}

The second estimate of this section will be used to control the behavior of the elements outside the bulk. Define
\[
Y_I = X_I \big(\log (X_I/\mathcal{D}) -1\big) + \mathcal{D} \, ,
\]
with the convention that $0 \cdot \log 0 = 0$. Note that $Y_I = \mathcal{I}(X_I)$, where $\mathcal{I}$ is the rate function of a Poisson random variable of mean $\mathcal{D}$. This implies that $Y_I \geq 0$ and vanishes only at $\mathcal{D}$. $\mathcal{I}$ is a convex function, and thus we can bound the sum of the $Y_I$'s by a function of the sum of the $X_I$'s, using Jensen's inequality. Furthermore, $\mathbb{P}[Y_I > t]$ should vanish as $\exp(-t)$, by ``inverting'' the rate function. We formalize this notion in the lemma below:
\begin{lem}\label{ExpMomentGeneratingY}
For any $\mathcal{D}$, and any positive $\lambda <1$,
\[
\mathbb{E}\left[\exp(\lambda Y_I)\right] \leq \frac{1+ \lambda}{1 - \lambda}.
\]
\end{lem}
\begin{proof}
The function $\mathcal{I}(x) = x [\log (x/\mathcal{D}) - 1] + \mathcal{D}$ is not invertible, but is piecewise invertible. First, let
\[
g_1(x): [0,\mathcal{D}] \rightarrow [0,\mathcal{D}] \text{ be a function such that } (\mathcal{I} \circ g_1) (x) = x\, .
\]
Note that this function is decreasing, with $g_1(0) = \mathcal{D}$ and $g_1(\mathcal{D}) = 0$. For any $x > \mathcal{D}$, we say that $g_1(x) = -\infty$. We define $g_2$, the second inverse, similarly, except its domain is defined to be $(\mathcal{D},\infty)$. This inverse is strictly increasing. Thus,
\[
\mathbb{P}[Y_I \geq t] =  \mathbb{P}[X_I \leq g_1 (t)] +  \mathbb{P}[X_I \geq g_2 (t)] \, . 
\]
By appealing to the two bounds of Lemma~\ref{ChernoffBounds}, we find that both the probabilities above are bounded above by $\exp(-t)$; in fact, if $t > \mathcal{D}$, the first probability is identically zero. Regardless, it will suffice to use the bound $\mathbb{P}[Y_I > t] < 2 e^{-t}$. Thus, for any $\lambda <1$,
\begin{align*}
\mathbb{E}[\exp(\lambda Y_I)] & = 1+\int_1^\infty \mathbb{P}\left[Y_I > \frac{\log t}{\lambda}\right] dt \\ & \leq 1+ 2 \int_1^\infty t^{-1/\lambda} dt \\ & = \frac{1+\lambda}{1 - \lambda}, 
\end{align*}
as required.
\end{proof}

We now uses the lemma to control the upper tail of the sum of the $Y_I$'s over {\em any} sufficiently small subset of $T$. We will only apply the proposition below on the set $\{1, 2, \dots, \mathcal{T}_M\}$ (which will be small with good probability from Proposition \ref{CardinalityBoundPro}), but it is actually more straightforward to consider the existence of a subset with bad properties, in order to avoid conditional probabilities. 

\begin{pro}\label{YIBound}
Let $Y_I$ as above, and $\alpha$ as in Proposition \ref{CardinalityBoundPro}.  Define the event
\[
\mathcal{B}_\xi := \biggl\{ \exists \mathfrak{W} \subset T, |\mathfrak{W}| \leq n^{\alpha} \text{ such that } \displaystyle \sum_{I \in \mathfrak{W}} Y_I > q (\log(q/w) -1) + \xi q  \biggr\}\, .
\]
Then, for all sufficiently large $n$,
\[
\mathbb{P}[\mathcal{B}_\xi] \leq \exp( - q [\log(q/w) -1 + \xi/2] ) \, .
\]
\end{pro}
\begin{proof}
Set $t = q (\log(q/w) -1 + \xi)$. Fix $\mathfrak{W} \subset T$ with cardinality at most $n^{\alpha}$. By Chebyshev's inequality
\begin{align*}
\mathbb{P}\Big[\displaystyle \sum_{I \in \mathfrak{W}} Y_I > t\Big] & \leq \frac{\big(\exp(\lambda Y_I)\big)^{n^\alpha}}{\exp(\lambda t)} \\ &\leq \left(\frac{1+ \lambda}{1 - \lambda} \right)^{n^\alpha} \cdot e^{-\lambda t} \, ,
\end{align*}
where the second inequality is Lemma~\ref{ExpMomentGeneratingY}.

We now set $\lambda = 1 - n^{\alpha} /{t}$, noting that, for sufficiently large $n$, $\lambda >0$ (since $n^{\alpha} << q$). This turns the above estimate into 
\begin{align*}
\mathbb{P}\Big[\displaystyle \sum_{I \in \mathfrak{W}} Y_I > t\Big] & \leq \left( 2t / n^{\alpha} \right)^{n^{\alpha}} \cdot e^{-t + n^{\alpha}} \\ & \leq \exp\left(- t + C n^{\alpha} \log n\right)\, .
\end{align*}
The final step is to apply the union bound:
\begin{align*}
\mathbb{P}[\mathcal{B}] & \leq \displaystyle \sum_{k=1}^{\lfloor n^{\alpha} \rfloor} {m^d \choose k} \cdot e^{-t + C n^{\alpha} \log n} \\ 
& \leq n^{\alpha} {m^d \choose n^{\alpha}} \cdot e^{-t + C n^{\alpha} \log n} \\ 
& \leq m^{d (n^{\alpha} +1)} \cdot e^{-t + C n^{\alpha} \log n}  \, . 
\end{align*}
Recalling Lemma \ref{lmm3}, we see that the combinatorial term in the final inequality is bounded above by $\exp( C n^{\alpha} \log n)$ for some (probably different) $C$. Since $q >> n^\alpha \log n$, the entire positive contribution can be bounded above by $\xi q/2$. This completes the proof.
\end{proof}

\subsection{Concentration Inequalities}
This section will prove concentration of the edge count of the random geometric graph {\em restricted to the bulk}. Explicitly, let 
\[
\hat{X}_I := X_I \cdot  1_{X_I < M}
\]
and $\hat{|E_s|}$ be the define analogously with $|E_s|$ by replacing $X_I$ with its truncated version (recall that $M = \lceil \mathcal{D} \rceil \cdot n^a$). In other words, $\hat{|E_s|}$ is the version of the edge count of $G_s$ obtained after deleting all vertices lying in $A_I$'s that satisfy $X_I \geq M$.

For the rest of the paper, fix 
\[
\beta = p - 2a.
\]
Consider the event 
\begin{equation*}
\mathcal{C} = \big\{ |\hat{E}_s| - \tilde{\mu}_s > n^\beta\big\}\, .
\end{equation*}
We control the probability of $\mathcal{C}$ in two regimes. We begin by assuming that $\mathcal{D} < \log n$.

Our strategy for proving an upper bound on $\mathcal{C}$ in this regime relies on Talagrand's convex concentration inequality \cite[Theorem 4.1.1]{Talagrand2}. First, let us define the setting: let $\Omega = \prod_{i=1}^N \Omega_i$, where $\Omega_i$ are all probability spaces and the measure $\mathbb{P}$ on $\Omega$ is the product measure. For a set $A \subset \Omega$, define the set
\[
U_A(x):= \{ \{s_i\} \in \{0,1\}^N : \exists y \in A, s_i = 0 \implies x_i = y_i\}\, .
\]
Let $V_A(x)$ be the convex hull of $U_A(x)$, and $d_c(A,x)$ is the $\ell^2$ distance of $V_A(x)$ to the origin. For any set $A$, we denote $A_t$ be the $t$ blowup of $A$ with respect to this metric, i.e.
\[
A_t := \{ x \in \Omega: d_c(A,x) \leq t\}\,.
\]
We can now state the inequality: 
\begin{theorem}[Talagrand's Inequality \cite{Talagrand2}]
If $\Omega$, $\mathbb{P}[\cdot]$, $A$ and $A_t$ are as above, then
\[
\mathbb{P}[A] \left(1 - \mathbb{P}[A_t]\right) \leq e^{-t^2/4} \, .
\]
\end{theorem}
We will not apply this theorem directly; instead, we use a corollary of this theorem frequently used in discrete settings \cite[Theorem 7.7.1]{Spencer}. To do so, we consider a random variable $X$ defined on the space $\Omega$, and a function $f$ from the natural numbers to the natural numbers. We say that $f$ is a witness function for $X$ if, whenever $X(\omega) \geq t$, there exists $I \subset [n]$ with $|I| \leq f(t)$, such that every $\omega'$ that agrees with $\omega$ in all $i \in I$ has $X(\omega') \geq t$. Furthermore, we assume that $X(\omega)$ is $K$-Lipschitz with respect to the Hamming distance --- that is, $|X(\omega) - X(\omega')| \leq K$ whenever $\omega$ and $\omega'$ differ in at most one coordinate.
\begin{theorem}[\cite{Spencer}]\label{DiscreteTalagrand}
Let $\Omega$ be a product space, and $X$ a real valued function on $\Omega$ with Lipschitz constant $K$ with respect to the Hamming distance. If $f$ is witness function for $X$ as above, then, for any $b$ and $t$,
\[
\mathbb{P}[X > b + t K\sqrt{f(b)}] \, \mathbb{P}[X \le b ] \leq \exp (-t^2/4)\, .
\]
\end{theorem}
With this preliminary complete, we will prove the following lemma:
\begin{lem}\label{ConcIneqSmall}
Let $\mathcal{C}$ be as above, and assume that $\mathcal{D} < \log n$. Then, for all sufficiently large $n$,
\[
\mathbb{P}[\mathcal{C}] \leq \exp\left( -  n^{2 \beta - p - 6a}\right) \, . 
\]
\end{lem}
\begin{proof}
Thanks to our assumption on $\mathcal{D}$, $M < n^a \log n$. 
We now apply Theorem~\ref{DiscreteTalagrand} to $X=\hat{|E_s|}$, considered as a function of the $X_I$'s. Since each coordinate is bounded above by $n^a \log n$, $X$ is Lipschitz with $K \leq \NI n^{2a} \log^2 n$. The function $f(w) = 2w$ is a  witness function for $\hat{|E_s|}$; to see this, note that $\hat{|E_s|}$ is the edge count of the $s$-graded geometric random graph, after we remove any $X_I$ with very high density. As such, we can ``witness'' the existence of $w$ edges by finding at most $2w$ vertices; the flexibility of the setup allows us to pick these vertices judiciously, avoiding all the isolated ones. Finding $2w$ vertices will require at most $2w$ distinct coordinates, if each one of them vertices lies in a distinct~$A_I$. Note that this bound may be very loose -- whenever $2w > m^d$, we can easily just check {\em every} $A_I$ to witness $\hat{|E_s|} > w$.

We apply the theorem with $b = \tilde{\mu}_s + n^{\beta}/\log n $ and
\[
t = \frac{n^{\beta}(1 - 1/\log n)}{ \NI n^{2a} \log^2 n \cdot  [2\tilde{\mu}_s + 2 n^\beta/ \log n]^{1/2}}\, .
\]
We deduce that 
\begin{align}\label{eq:TalagrandBoundSmall}
\mathbb{P}[\mathcal{C}] \cdot  \mathbb{P}[\hat{|E_s|}  \leq \tilde{\mu}_s + n^{\beta}/\log n] & \leq \exp\left(-\frac{n^{2\beta}\left[1 - 1/\log n\right]^2 }{8 \NI^2  n^{4a}\log^4 n [\tilde{\mu}_s + n^{\beta}/\log n]} \right) \nonumber \\ & \leq  \exp\left(-n^{2 \beta - p -5a} \right)\, ,
\end{align}
where the final inequality holds for sufficiently large $n$, using the fact that $\tilde{\mu}_s = n^{p +o(1)} >> n^{\beta}/\log n$ and the fact that $\NI$ is uniformly bounded in $n$. 

To complete the proof, we must show that $\mathbb{P}[\hat{|E_s|}  \leq \tilde{\mu}_s + n^{\beta}/\log n]$ is not too small. Since the mean of $\hat{|E_s|}$ is strictly smaller than the mean of $|E_s|$, it is enough to show that
\[
\mathbb{P}[\hat{|E_s|} - \mathbb{E}[\hat{|E_s|}] \geq n^{\beta}/\log n] < \varepsilon.
\]
We will produce a very crude bound on the variance of $\hat{|E_s|}$: let $\hat{Z}_I = \hat{X}_I \left(\sum_{J \in \tilde{N}_I} \hat{X}_J - 1 \right)$. Then clearly,
\begin{align*}
\text{Var}[\hat{|E_s|}] & = \sum_{I,J} \mathbb{E}\left[\left(\hat{Z}_I - \mathbb{E}[\hat{Z}_I]\right)\left(\hat{Z}_J - \mathbb{E}[\hat{Z}_J]\right) \right] \\ & \leq \sum_{I} |\tilde{S}_I| \cdot \mathbb{E}\left[\left(\hat{Z}_I - \mathbb{E}[\hat{Z}_I]\right)^2\right]
\end{align*}
A straightforward computation will show that, for some constant $C$ independent of $n$, 
\[
\mathbb{E}\left[\left(\hat{Z}_I - \mathbb{E}[\hat{Z}_I]\right)^2 \right] \leq C (\mathcal{D}^2 + \mathcal{D}^4).
\]
Since $\mathcal{D} < \log n$ and $|\tilde{S}_I|$ is uniformly bounded in $n$ (from Lemma \ref{lmm3}), this implies that $\text{Var}[\hat{|E_s|}] < n^{p +o(1)}$, and Chebyshev's inequality gives that 
\[
\mathbb{P}[\hat{|E_s|} - \mathbb{E}[\hat{|E_s|}] \geq n^{\beta}/\log n] \leq n^{p - 2\beta + o(1)}.
\]
For any admissible value of $p$, this function vanishes as $n$ increases, and $\mathbb{P}[\hat{|E_s|}  \leq \tilde{\mu}_s + n^{\beta}/\log n] > 1 - \varepsilon$. Substituting this into \eqref{eq:TalagrandBoundSmall} gives us
\[
\mathbb{P}[\mathcal{C}] \leq \exp\left(- n^{2 \beta - p -6a} \right) \, ,
\]
completing the proof.
\end{proof}
Next, assume that $\mathcal{D} \geq \log n$. In this regime, we replace Talagrand's inequality with the celebrated Azuma--Hoeffding inequality: 
\begin{theorem}[Azuma--Hoeffding inequality~\cite{azuma, hoeffding63}]
Let $\{Z_0, Z_1, \dots ,Z_n\}$ be a martingale sequence with $|Z_{k} - Z_{k-1}| < c_k$ for all $k$. Then
\[
\mathbb{P}\left[|Z_n - Z_0| > t\right] \leq 2 \exp\left(- \frac{t^2}{2 \sum_{k=1}^n c_k^2}\right) \, .
\]
\end{theorem}

We wish to prove the following lemma, bounding the probability of the event $\mathcal{C}$:
\begin{lem}\label{ConcIneqLarge}
Assume that $\mathcal{D} \geq \log n$. Then, for all $n$ sufficiently large,
\[
\mathbb{P}[\mathcal{C}] \leq \exp\left(-\frac{ n^{2\beta}}{m^d \mathcal{D}^3 n^{6a}} \right) \, .
\]
\end{lem}
We note that a naive application of Azuma--Hoeffding to the martingale given by conditioning on the value of $\hat{X}_I$ would give a bound on the probability of $\mathcal{C}$ which depends on the {\em fourth} power of $\mathcal{D}^{-1}$, not the third as in the lemma --- an inferior bound. Thus, we need to be more careful in this analysis.
\begin{proof}
We partition $A_I$ into sets of measure $1/n$; formally, let $\{F_{I,t}\}$, for natural $t \leq \lceil \mathcal{D} \rceil$ be a collection of disjoint subsets of $A_I$ such that $\lambda(F_{I,t}) = 1/n$ for every $t \leq \lceil \mathcal{D} \rceil -1$, and
\[
\bigcup_{t} F_{I,t} = A_I \, .
\]
Note that the measure of the final $F_{I,t}$ will be strictly smaller than $1/n$, unless $\mathcal{D}$ is an integer. We define $W_{I,t} = |\chi(F_{I,t})|$.

Clearly, $\sum_{t} W_{I,t} = X_I$. Define $\overline{|E_s|}$ as (yet another!) truncation of $|E_s|$. Specifically, let $\overline{W_{I,t}} = W_{I,t} \cdot 1_{W_{I,t} < n^a/2}$, and define $\overline{|E_s|}$ by replacing each $X_I$ in the definition of $|E_s|$ by $\sum_t \overline{W_{I,t}}$. Note that $\overline{|E_s|}$ is a function of $m^d \lceil \mathcal{D}  \rceil$ independent random variables. Letting $\mathcal{F}'_\ell$ be the $\sigma$-algebra generated by the first $\ell$ $W_{I,t}$'s (enumerated arbitrarily), we once again have a martingale sequence $Z'_\ell = \mathbb{E}\left[\overline{|E_s|} \, \big| \, \mathcal{F}'_\ell\right]$. We also have that
\[
|Z'_{\ell} - Z'_{\ell-1}| \leq \NI \cdot \lceil \mathcal{D} \rceil n^{2a} \, .
\] 
Thus, Azuma--Hoeffding implies that
\begin{align*}
\mathbb{P}\left[\overline{|E_s|} - \tilde{\mu}_s > n^{\beta}/2 \right] & \leq \mathbb{P}\left[\overline{|E_s|} - \mathbb{E}[\overline{|E_s|}] > n^{\beta}/2 \right] \\ & \leq 2\exp\left(- \frac{n^{2\beta}}{8 m^d \lceil\mathcal{D}\rceil^3  \NI ^2 n^{4a}}\right) \, ,
\end{align*}
where the first line follows since $\tilde{\mu}_s > \mathbb{E}[\overline{|E_s|}]$. Using the uniform bound on $\NI$, we deduce that 
\begin{align*}
\mathbb{P}\left[\overline{|E_s|} - \tilde{\mu}_s > n^{\beta}/2 \right] & \leq \exp\left(-\frac{n^{2\beta + o(1)}}{m^d \mathcal{D}^3 n^{4a}} \right) \\ &  \leq \exp\left(-\frac{n^{2\beta}}{m^d \mathcal{D}^3 n^{5a}} \right) \, .
\end{align*}
By partitioning, 
\begin{equation}\label{CPartitioning}
\mathbb{P}[\mathcal{C}] \leq \mathbb{P}[\overline{|E_s|} - \tilde{\mu}_s > n^{\beta}/2] + \mathbb{P}[ \mathcal{C}, \overline{|E_s|} - \tilde{\mu}_s < n^{\beta}/2] \, .
\end{equation}
The second event on the righthand side implies the event
\[
\mathcal{E} := \{|\hat{E}_s| - \overline{|E_s|} > n^{\beta}/2 \} \, . 
\]
The lemma will follow if we can produce a good upper bound on the probability of the event $\mathcal{E}$.

The difference between the random variables $|\hat{E}_s|$ and $\overline{|E_s|}$ is given by configurations in which at least $W_{I,t}$ is larger than $n^{a}/2$. In fact,
\[
\hat{|E_s|} - \overline{|E_s|} < \displaystyle \sum_{(I,t)} \left[ \left( W_{I,t} \cdot 1_{W_{I,t} > n^a/2} \right) \cdot \sum_{J : \rho(I,J) \leq s} \left( X_J \cdot 1_{X_J < \lceil \mathcal{D} \rceil n^a} \right) \right]\, .
\]
While the random variables in the expression above are far from independent, we can replace the second sum over the $X_J$'s by $\NI \cdot \lceil \mathcal{D} \rceil n^a$, the upper bound imposed on it by the indicator random variables involved. Therefore,
\[
\mathbb{P}[\mathcal{E}] \leq \mathbb{P}\left[\displaystyle \sum_{(I,t)} \left(W_{I,t} \cdot 1_{W_{I,t} > n^a/2} \right) > \frac{C n^\beta}{\mathcal{D} n^{a}} \right]
\]
To bound this final probability, we can directly bound the exponential moment of $W_{I,t} \cdot 1_{W_{I,t} > n^a/2}$:
\[
\mathbb{E}\left[\exp\left(W_{I,t} \cdot 1_{W_{I,t} > n^a/2} \right) \right] \leq 1 + \displaystyle \sum_{k > n^a/2} \frac{e^k}{k!} \leq 1 + \exp( - n^a) \, .
\]
The first inequality follows because $W_{I,t}$ is a Poisson random variable of mean $1$ (or possibly less than $1$, if we pick the small $W_{I,t}$ in each $I$), while the second can deduced by using Stirling's approximation and explicitly summing.

Applying a Chernoff strategy, we find that
\[
\mathbb{P}\left[\displaystyle \sum_{(I,t)} \left(W_{I,t} \cdot 1_{W_{I,t} > n^a/2} \right) > \frac{C n^\beta}{\mathcal{D} n^{a}} \right] \leq \left(1 + \exp(-n^a) \right)^{m^d (\mathcal{D} +1)} \cdot \, \exp\left( - \frac{C n^\beta}{\mathcal{D} n^{a}}\right) \, .
\]
Using the standard approximation $(1+x) \leq e^{x}$, we find that the prefactor is bounded by $2$ for all $n$ sufficiently large.

Substituting the bounds into \eqref{CPartitioning} gives
\[
\mathbb{P}[\mathcal{C}] \leq \exp\left(- \frac{n^{2\beta}}{m^d \mathcal{D}^3 n^{5a}}\right) + 2 \exp\left( - \frac{C n^\beta}{\mathcal{D} n^{a}}\right) \, .
\]
The first term vanishes like $\exp(-n^{1 - 9a + o(1)})$, whereas the second vanishes as $\exp( - n^{1-7a +o(1)} )$. Therefore, we conclude that 
\[
\mathbb{P}[\mathcal{C}] \leq \exp\left(- \frac{n^{2\beta}}{m^d \mathcal{D}^3 n^{6a}}\right),
\]
as required.
\end{proof}

\subsection{Reducing to Deterministic Inequalities}
The final probabilistic step of this proof involves bounding the probability of a rather complicated set of simultaneous inequalities. Luckily, instead of dealing with the event itself, we will control it using the events $\mathcal{A}$, $\mathcal{B}_\xi$, and $\mathcal{C}$, whose probability we controlled previously.

Recall that
\[
V(\mathfrak{W}) = \frac{1}{q} \displaystyle \sum_{I \in \mathfrak{W}} X_I \;
\]
We also recall from the outline that, for any $\mathfrak{W} \subset T$, we set
\begin{equation}\label{Qdef}
Q(\mathfrak{W}) := \frac{2}{q^2}\sum_{I \in \mathfrak{W}} \biggl({X_I \choose 2} + \frac{1}{2} \displaystyle \sum_{\substack{J \in \tilde{N}_I\cap \mathfrak{W}\\J\ne I}} X_I X_J \biggr)\, .
\end{equation}
This counts the number of edges with both endpoints in $A_I$'s with $I \in\mathfrak{W}$, normalized b $q^2/2$; this choice will avoid many unnecessary factors in our later analysis. This immediately implies that, for any $\mathfrak{W} \subset T$,  
\[
Q(\mathfrak{W}) \leq \left[V(\mathfrak{W})\right]^2 \, .
\]
Using this notation, we can formulate Jensen's inequality in the following way:
\begin{lem}\label{lem:Jensen}
For any $\mathfrak{W} \subset T$,
\[
\frac{1}{q} \displaystyle \sum_{I \in \mathfrak{W}} Y_I \geq V(\mathfrak{W}) \left[\log \left(\frac{q}{w}\right) + \log V(\mathfrak{W}) - \log \left(\frac{|\mathfrak{W}|}{\tilde{\tau}_s}\right)- 1\right]\,.
\]
\end{lem}
\begin{proof}
This is  a direct application of Jensen's inequality to the convex function $ Y_I = X_I \big(\log (X_I/\mathcal{D}) -1\big) + \mathcal{D}$. It implies that
\[
\displaystyle \sum_{I \in \mathfrak{W}} Y_I \geq q V(\mathfrak{W}) \left[ \log \left(\frac{q V(\mathfrak{W})}{|\mathfrak{W}| \mathcal{D}}\right) -1 \right] + \mathcal{D} | \mathfrak{W}|.
\]
Dividing through by $q$, using the definition of $w$, and ignoring the positive additive term $\mathcal{D}|\mathfrak{W}|/q$ gives the bound above.
\end{proof}

We now begin the part of the analysis where we will obtain the universal constant $\teps_0$ in the statement of Theorem \ref{thm2}. For that, we first introduce a parameter $\xi>0$. We will eventually define $\teps$ in terms of $\xi$. In the following, wherever we say ``$\xi$ sufficiently small'', we mean ``$\xi\le \xi_0$ for some universal constant $\xi_0$''. The meaning of ``$n$ sufficiently large'' will remain as it is.

Recal that $\mathcal{T}_M := \max \{I : X_I \geq M\}$, and let $\mathfrak{I} := \{1, 2,\dots, \mathcal{T}_M\}$. Given $\xi >0$ and a constant $C$ which does not depend on $n$, we define $\mathcal{H}_\xi$ to be the event that the following four inequalities hold: 
\begin{align}
&\mathcal{T}_M  < n^{\alpha}\,, \label{bound1}\\  &\frac{1}{q}  \sum_{I \in \mathfrak{I}} Y_I \leq \log (q/w) - 1 + \xi\, , \label{bound2}\\
&V(\mathfrak{I})  < C,   \label{bound3} \\ 
& Q(\mathfrak{I}) \geq 1 - \frac{\xi}{\log n}\, .\label{bound4}
\end{align}
\begin{pro}\label{prop:DiscreteStructureTheroem}
Define $\mathcal{H}_{\xi}$ as above. Then, for all $n$ sufficiently large and $\xi$ sufficiently small, 
\[
\mathbb{P}[\mathcal{H}_{\xi}^c \cap \mathcal{L}_n(\tilde{\delta})] \leq 3\exp( - q [\log(q/w) -1 + \xi/2] ).
\]
\end{pro}
\begin{proof}
Suppose that $\mathcal{A}^c \cap \mathcal{B}_\xi^c \cap \mathcal{C}^c \cap \mathcal{L}_n(\tilde{\delta})$ implies $\mathcal{H}_\xi $. If this were true, then the union bound would imply that 
\[
\mathbb{P}[\mathcal{H}_{\xi}^c \cap \mathcal{L}_n(\tilde{\delta})] \leq  \mathbb{P}[\mathcal{A}] + \mathbb{P}[\mathcal{B}_\xi] +\mathbb{P}[\mathcal{C}]. 
\]
The righthand side of the proposition above can be expressed as $e^{-n^{p/2 + o(1)}}$. Therefore, for all sufficiently large $n$, Proposition \ref{CardinalityBoundPro} implies that
\[
\mathbb{P}[\mathcal{A}]\cdot \exp(q [\log(q/w) -1 + \xi/2] ) \leq \exp\left(- n^{p/2 + a/3} + n^{p/2 +o(1)} \right) < 1,
\]
where the final inequality follows because the first exponent is larger than $p/2$ for all admissible values of $p$. Similar analysis of exponents and Lemma \ref{ConcIneqSmall} tell us that, if $\mathcal{D} < \log n$ and $n$ is large,
\[
\mathbb{P}[\mathcal{C}]\cdot \exp(  q [\log(q/w) -1 + \xi/2] \leq \exp\left(- n^{2 \beta - p - 6a} + n^{p/2 +o(1)} \right) < 1.
\]
When $\mathcal{D} > \log n$, $\mathbb{P}[\mathcal{C}] < \exp( - n^{1 - 10a - o(1)})$ by Lemma \ref{ConcIneqLarge}, and, for all $n$ sufficiently large, 
\[
\mathbb{P}[\mathcal{C}]\cdot \exp( q [\log(q/w) -1 + \xi/2] ) \leq \exp\left(- n^{1 - 10a} + n^{p/2 +o(1)} \right) < 1.
\]
Finally, Proposition \ref{YIBound} bounds the probability of $\mathcal{B}_\xi$ by the righthand side. We conclude that, if $\mathcal{A}^c \cap \mathcal{B}_\xi^c \cap \mathcal{C}^c \cap \mathcal{L}_n(\tilde{\delta})$ implies $\mathcal{H}_\xi$,
\[
\mathbb{P}[\mathcal{H}_{\xi}^c \cap \mathcal{L}_n(\tilde{\delta})] \leq 3\exp( - q [\log(q/w) -1 + \xi/2] ).
\]
For the remainder of this proof, we condition on the event $\mathcal{A}^c \cap \mathcal{B}_\xi^c \cap \mathcal{C}^c \cap \mathcal{L}_n(\tilde{\delta})$. The event $\mathcal{A}^c$ automatically implies \eqref{bound1}. With this bound on $\mathcal{T}_M$, \eqref{bound2} is immediate from $\mathcal{B}_\xi^c$.  Next, we apply Lemma~\ref{lem:Jensen} to the sum of the $Y_I$'s over $I \in \mathfrak{I}$ to deduce that
\[
V(\mathfrak{I}) \left[\log \left(\frac{q \tilde{\tau}_s}{w \cdot \mathcal{T}_M}\right)  - 1\right] + V(\mathfrak{I})  \log V(\mathfrak{I}) \leq  \log (q/w) - 1 + \xi\, .
\]
If $V(\mathfrak{I}) \leq 1$, we have a (much better than needed!) bound on $V(\mathfrak{I})$. Otherwise, $V(\mathfrak{I})  \log V(\mathfrak{I})$ is positive, and we conclude that
\begin{equation*}
V(\mathfrak{I}) \leq \frac{\log (q/w) - 1 + \xi }{\log (q \tilde{\tau}_s/(w \cdot \mathcal{T}_M)) -1}\, .
\end{equation*}
By \eqref{bound1} and the definitions of the variables $q$, $w$ and $\alpha$,
\begin{equation}\label{eq:PolyLowerBound}
\frac{q \tilde{\tau}_s }{w \cdot \mathcal{T}_M} \geq n^{a/2 + o(1)} \, .
\end{equation}
Therefore, the denominator grows at least as a constant multiple of $\log n$. Meanwhile, $q/w \leq n^{1-p/2 + o(1)}$. This proves \eqref{bound3}.

For the final stipulation, we use the event $\mathcal{L}_n(\tilde{\delta}) \cap \mathcal{C}^c$. By $\mathcal{C}^c$,
\begin{equation}\label{ConcBound}
(q^2 /2)\cdot Q(\mathfrak{I}^c) - \tilde{\mu}_s \leq n^\beta\, .
\end{equation}
By the occurrence of $\mathcal{L}_n(\tilde{\delta})$,
\[
\frac{q^2}{2} \big(Q(\mathfrak{I}) + Q(\mathfrak{I}^c)\big) + \left(\sum_{I \in \mathfrak{I}} \displaystyle \sum_{J \in \tilde{N}_I \cap \mathfrak{I}^c} X_I X_J \right)\geq (1 + \tilde{\delta}) \tilde{\mu}_s \, ,
\]
where the first term counts edges with both endpoints either in or outside the bulk, whereas the sum counts the number of edges with exactly one endpoint in the bulk. Using the upper bound on $Q(\mathfrak{I}^c)$ given by \eqref{ConcBound}, we deduce that
\begin{equation}\label{Qlb}
Q(\mathfrak{I}) + \frac{2}{q^2} \left(\sum_{I \in \mathfrak{I}} \displaystyle \sum_{J \in \tilde{N}_I \cap \mathfrak{I}^c} X_I X_J \right)\geq 1 - \frac{2n^\beta}{q^2} \, .
\end{equation}
By definition, $X_J \leq M$ whenever $J \not \in \mathfrak{I}$, and therefore 
\[
\frac{2}{q^2} \left(\sum_{I \in \mathfrak{I}} \displaystyle \sum_{J \in \tilde{N}_I \cap \mathfrak{I}^c} X_I X_J \right) \leq \frac{2 \NI  M}{q} \cdot V(\mathfrak{I}) \, .
\]
Recalling \eqref{eq:DefinitionofM}, we see that 
\[
\frac{M}{q} = n^{\max \{a - p/2 + o(1), p/2 - 1 + a + o(1)\} }.
\]
The exponent is always negative, and therefore, for any admissible $s$-graded model, $M/q < \xi/(2 C \log n)$ for sufficiently large $n$. By \eqref{bound3} and Lemma \ref{lmm3}, $V(\mathfrak{I})$ and $|\tilde{N}_I|$ are uniformly bounded in $n$. Similarly, we can increase $n$ sufficiently to ensure that $(2 n^\beta)/q^2$ is also bounded above by $\xi/(2 \log n)$ (which is possible thanks to the definition of $\beta$). Substituting the bounds into \eqref{Qlb} shows that $\mathcal{A}^c \cap \mathcal{B}_\xi^c \cap \mathcal{C}^c \cap \mathcal{L}_n(\tilde{\delta})$ implies $\mathcal{H}_\xi$, completing the proof. 
\end{proof}

\subsection{Controlling the Linear Sum}
The rest of the section is dedicated to analyzing configurations in $\mathcal{H}_\xi$. We will condition on this event --- i.e.~assume that the four inequalities in the definition hold, and show that subsets with certain properties exist. We emphasize that the statement will hold for {\em any} positive integer $s$, asymptotically in $n$. The number $\xi$ will be chosen to be sufficiently small for certain estimates to hold.

Define $\mathcal{T}_V$ by
\begin{equation}\label{eq:mathcalTV}
\mathcal{T}_V := \min \left\{ k : V(\{1, \dots, k\}) > 1 - \frac{2\xi}{\log n} \right\} \, . 
\end{equation} 
{\em A priori}, the set $\mathfrak{T}:= \{1 ,2, \dots, \mathcal{T}_V\}$ may include some elements of the bulk. The next lemma proves that not only are all these indices away from bulk, but, in fact, restricting our attention to $\mathfrak{T}$ does not force us to ignore too many edges.

\begin{lem}\label{lmm6}
Define $\mathfrak{T}$ as above, and assume that $\mathcal{H}_\xi$ holds. Then, for all $n$ sufficiently large and $\xi$ sufficiently small, the following holds:
\[
\tilde{\tau}_s\left(1 - \xi^{1/2}\right) \leq \mathcal{T}_V \leq \mathcal{T}_M.
\]
\begin{align*}
Q(\mathfrak{T}) &\geq 1 - \psi(\mathcal{T}_V)\,
\end{align*}
and
\begin{align*}
 1 - \frac{2\xi}{\log n} &\leq V(\mathfrak{T}) \leq 1 + \phi(\mathcal{T}_V)
\end{align*}
where
\[
\phi(x) = \min \left\{\frac{C[\log(x/ \tilde{\tau}_s) + \xi]}{\log n}, \, \frac{2}{x}\right\}
\]
and
\[
\psi(x) = \min \left\{ \frac{C[1 + \log(x/ \tilde{\tau}_s)] }{\log n},\, \frac{C'}{x} \right\} + \frac{\xi}{\log n}
\]
for some constants $C$ and $C'$ independent of $n$.
\end{lem}
The exact forms of $\phi$ and $\psi$ are chosen to make the proof more transparent. The important feature of the functions are that $\phi$ and $\psi$ decrease for large $x$, providing better bounds whenever $\mathcal{T}_V$ is large. Since we have no {\em a priori} bound for this cardinality, this will be crucial for later analysis. Furthermore, for any positive $x$ and sufficiently large value of $n$,
\begin{equation}\label{eq:UBpsiphi}
\max\{ \psi(x), \phi(x)\} \leq \frac{1}{\sqrt{\log n}}. 
\end{equation} 

\begin{proof}
Since $Q(\mathfrak{I}) \geq 1- \xi/\log n$ (by \eqref{bound4}), we know that 
\[
V(\mathfrak{I}) \geq  \sqrt{Q(\mathfrak{I})} \geq 1- \xi/\log n \, ,
\]
which immediately implies $\mathcal{T}_V \leq \mathcal{T}_M$. Since $Y_I \geq 0$, the upper bound \eqref{bound2} in the definition of $\mathcal{H}_\xi$ can be applied to elements of $\mathfrak{T}$. Applying Lemma~\ref{lem:Jensen} to this set, we deduce that
\begin{equation}\label{JensenT}
V(\mathfrak{T}) \left[\log \left(\frac{q}{w}\right) + \log V(\mathfrak{T}) - \log \left(\frac{\mathcal{T}_V}{\tilde{\tau}_s}\right) - 1\right] \leq \log \left(\frac{q}{w}\right) + \xi -1 \, .
\end{equation}
By definition, $V(\mathfrak{T})$ is at least $1 - 2\xi/ \log n$. Noting that 
\[
\log\left(1 - \frac{2 \xi}{\log n}\right) \geq \frac{-4 \xi}{\log n}
\]
for all sufficiently large $n$, we can conclude that
\[
- \left(1 - \frac{2 \xi}{\log n} \right) \log \left(\frac{\mathcal{T}_V}{\tilde{\tau}_s}\right) \leq \xi +  \frac{2 \xi \log (q/w)}{\log n} + \frac{C \xi}{\log n}\,.
\]
We recall that $q/w \leq n^{1-p/2 + o(1)}$, and therefore there exists a constant $C$ such that
\[
- \log \left(\frac{\mathcal{T}_V}{\tilde{\tau}_s}\right) \leq C \xi\, .
\]
Inverting the negative logarithm gives
\[
\mathcal{T}_V \geq \tilde{\tau}_s \exp(-C \xi) \geq \tilde{\tau}_s(1 - C \xi) \geq \tilde{\tau}_s (1 - \xi^{1/2})\, ,
\]
where the final inequality holds for all sufficiently small $\xi$, as required.

We prove the upper bound on $V(\mathfrak{T})$ by proving a bound that holds for all values of $\mathcal{T}_V$, and then improving it in the case $\mathcal{T}_V < \log n$. We observe that the definition of the ordering of the $X_I$'s implies that $X_{\mathcal{T}_V}$ is equal to the minimum of the set $\{X_1, X_2, \dots, X_{\mathcal{T}_V}\}$, and thus
\begin{equation}\label{eq:ubTV}
\frac{X_{\mathcal{T}_V}}{q} \leq \frac{V(\mathfrak{T})}{\mathcal{T}_V} \, .
\end{equation}
Furthermore, by minimality of $\mathcal{T}_V$,  
\[
V(\mathfrak{T} \setminus \{X_{\mathcal{T}_V}\}) \leq 1 - \frac{2 \xi}{\log n} \, .
\]
Recall that $\ttau \geq 2^d$ (from \eqref{eq:ttaubound}) and therefore, $\mathcal{T}_V > (1 - \xi^{1/2}) \ttau$ implies that $\mathcal{T}_V \geq 2$ for any $\xi$ sufficiently small. Since $V(\mathfrak{T}) = V(\mathfrak{T} \setminus \{X_{\mathcal{T}_V}\}) + X_{\mathcal{T}_V}/q$, we deduce that
\[
V(\mathfrak{T})\leq \frac{1 - (2 \xi)/(\log n)}{1 - 1/\mathcal{T}_V} \leq 1 +  \frac{2}{\mathcal{T}_V} - \frac{2 \xi}{\log n}\, .
\]
Ignoring the negative contribution $2 \xi/ \log n$ gives the desired bound.

Whenever $\mathcal{T}_V \geq \log n$, an explicit computation will show that $\phi(\mathcal{T}_V) = 2/ \mathcal{T}_V$ for all sufficiently large $n$. Thus, to complete the bound on $V(\mathfrak{T})$, we may assume that $\mathcal{T}_V < \log n$.
We now return to \eqref{JensenT}. If $V(\mathfrak{T}) \leq 1$, we are done. Otherwise, $V(\mathfrak{T}) \log V(\mathfrak{T})$ is positive, and thus
\begin{align*}
V(\mathfrak{T}) &\leq \frac{\log(q/w) - 1 + \xi}{\log(q/w) - \log \left(\mathcal{T}_V/\tilde{\tau}_s\right)- 1} \\ &  = 1 + \frac{\log \left(\mathcal{T}_V/\tilde{\tau}_s\right) + \xi}{ \log(q/w) - \log \left(\mathcal{T}_V/\tilde{\tau}_s\right)-1} \\ & \leq 1 + \frac{C[\log \left(\mathcal{T}_V/\tilde{\tau}_s\right) + \xi]}{\log n} \, ,
\end{align*}
where we use the upper bound $\mathcal{T}_V \leq \mathcal{T}_M$ and \eqref{eq:PolyLowerBound} to ensure that the denominator is bounded below by $[(a/3) \log n -1  ]> (a/4) \log n$ for all sufficently large $n$.  This gives half the desired upper bound on $V(\mathfrak{T})$ under the assumption $\mathcal{T}_V < \log n$.

The loewr bound on $Q(\mathfrak{T})$ will follow a similar strategy. We begin by noting that an algebraic manipulation will prove that, for any set $\mathfrak{W} \subset T$, 
\[
Q(\mathfrak{W}) = \frac{1}{q^2} \left( \sum_{I \in \mathfrak{W}} X_I \left[\sum_{J \in \tilde{N}_I \cap \mathfrak{W}} X_J \right] \right) - \frac{V(\mathfrak{W})}{q}. 
\]
Let $\mathfrak{Z} = \mathfrak{I} \setminus \mathfrak{T}$, and observe that 
\begin{align*}
Q(\mathfrak{I}) - Q(\mathfrak{T}) & \leq \frac{1}{q^2} \cdot \left(\displaystyle \sum_{I \in \mathfrak{I}} X_I \left[\sum_{J \in \tilde{N}_I \cap \mathfrak{I}} X_J \right] - \sum_{I \in \mathfrak{T}} X_I \left[\sum_{J \in \tilde{N}_I \cap \mathfrak{T}} X_J \right]\right) \, \\ &  = \frac{1}{q^2} \cdot \left(\displaystyle \sum_{I \in \mathfrak{Z}} X_I \left[\sum_{J \in \tilde{N}_I \cap \mathfrak{I}} X_J \right] + \sum_{I \in \mathfrak{T}} X_I \left[\sum_{J \in \tilde{N}_I \cap \mathfrak{Z}} X_J \right]\right)
\end{align*}
We decompose the first sum into 
\begin{align*}
\sum_{I \in \mathfrak{Z}} X_I \left[\sum_{J \in \tilde{N}_I \cap \mathfrak{I}} X_J \right] & = \sum_{I \in \mathfrak{Z}} X_I \left[\sum_{J \in \tilde{N}_I \cap \mathfrak{Z}} X_J \right] + \sum_{I \in \mathfrak{Z}} X_I \left[\sum_{J \in \tilde{N}_I \cap \mathfrak{T}} X_J \right] \\ & = \sum_{I \in \mathfrak{Z}} X_I \left[\sum_{J \in \tilde{N}_I \cap \mathfrak{Z}} X_J \right] + \sum_{I \in \mathfrak{T}} X_I \left[\sum_{J \in \tilde{N}_I \cap \mathfrak{Z}} X_J \right], 
\end{align*}
where we can flip the order of summation in the final equality thanks to the symmetry of $\rho$. Substituting this in, we find that 

\begin{equation}
Q(\mathfrak{I}) - Q(\mathfrak{T}) \leq \frac{2}{q^2} \left(\displaystyle \sum_{I \in \mathfrak{I}} X_I \left[\sum_{J \in \tilde{N}_I \cap \mathfrak{Z}} X_J \right] \right). \label{eq:breakapart}
\end{equation}
By ordering of the $X_I$'s, we have that $X_I \leq X_{\mathcal{T}_V}$ for any $I \in \mathfrak{Z}$. Therefore, 
\[
\frac{1}{q} \sum_{J \in \tilde{N}_I \cap \mathfrak{Z}} X_J \leq \frac{\NI  V(\mathfrak{T})}{\mathcal{T}_V}, 
\]
where we reuse \eqref{eq:ubTV}. Substituting this into \eqref{eq:breakapart}, we see that 
\[
Q(\mathfrak{I}) - Q(\mathfrak{T}) \leq \frac{ 2 \NI V(\mathfrak{I}) \cdot V(\mathfrak{T})}{\mathcal{T}_V}. 
\]
Thanks to $\mathcal{H}_\xi$, $V(\mathfrak{I})$ is uniformly bounded in $n$ (by \eqref{bound3}) and $Q(\mathfrak{I}) \geq 1 - \xi/\log n$ (by \eqref{bound4}). The uniform bound on $\NI$ from Lemma \ref{lmm3} proves half of the lower bound on $Q(\mathfrak{T})$.

Again, if $\mathcal{T}_V \geq \log n$, $\psi(\mathcal{T}_V) = C'/ \mathcal{T}_V + \xi/ \log n$ for all sufficiently large $n$. Thus, we may assume that $\mathcal{T}_V < \log n$ for the rest of the proof. Suppose that we were given the bound $V(\mathfrak{Z}) \leq  C[1 + \log(\mathcal{T}_V/\ttau)]/\log n$. From \eqref{eq:breakapart}, we konw that 
\begin{align*}
Q(\mathfrak{I}) - Q(\mathfrak{T}) &\leq 2 V(\mathfrak{I}) \cdot V(\mathfrak{Z}) \leq \frac{C\left[1 + \log \left(\frac{\mathcal{T}_V}{\ttau} \right)\right]}{\log n},
\end{align*}
where we use \eqref{bound3} to replace $V(\mathfrak{I})$ by a uniform constant. Thus, it is sufficient to prove the upper bound on $V(\mathfrak{Z})$.

To do so, we return to \eqref{bound4}, and apply Lemma \ref{lem:Jensen} to $\mathfrak{T}$ {\em without} ignoring the contibution of elements of $\mathfrak{Z}$. This allows us to deduce that
\begin{align}\label{eq:JensenParts}
V(\mathfrak{T}) & \left[\log \left(\frac{q}{w}\right) + \log V(\mathfrak{T}) - \log \left(\frac{\mathcal{T}_V}{\tilde{\tau}_s}\right) - 1\right]  + \frac{1}{q} \displaystyle \sum_{I \in \mathfrak{Z}} Y_I \, \leq \,  \log(q/w) - 1 + \xi\,.
\end{align}
Thanks to the upper bound on $\mathcal{T}_V$, we know that the bracketed term is positive and increasing in $V(\mathfrak{T})$. Since $V(\mathfrak{T}) \geq 1 - 2 \xi/ \log n$, some careful calculations imply that, for all sufficiently large $n$ and sufficiently small $\xi$, 
\begin{align*}
V(\mathfrak{T}) \left[\log \left(\frac{q}{w}\right) + \log V(\mathfrak{T}) - \log \left(\frac{\mathcal{T}_V}{\tilde{\tau}_s}\right) - 1\right] & \geq \log(q/w) - 1 - \log \left(\frac{\mathcal{T}_V}{\ttau}\right) \\ & \quad - \frac{3 \xi}{\log n} - \frac{3 \xi \log(q/w)}{\log n}.
\end{align*}
Substituting this into \eqref{eq:JensenParts} gives that 
\begin{align*}
\frac{1}{q} \displaystyle \sum_{I \in \mathfrak{Z}} Y_I & \leq \xi + \frac{3 \xi}{\log n} + \log \left(\frac{\mathcal{T}_V}{\tilde{\tau}_s}\right) + \frac{3 \xi \log(q/w)}{\log n} \\ & \leq C \xi  + \log \left(\frac{\mathcal{T}_V}{\tilde{\tau}_s}\right),
\end{align*}
for some uniform $C$ and all sufficiently large $n$. Another application of Lemma \ref{lem:Jensen} --- this time, to $\mathfrak{Z}$ ---  implies that 
\[
V(\mathfrak{Z}) \log\left(\frac{q \tilde{\tau_s}}{w \cdot \mathcal{T}_M} \right) + V(\mathfrak{Z}) (\log V(\mathfrak{Z}) -1) \leq C \xi +  \log \left(\frac{\mathcal{T}_V}{\tilde{\tau}_s}\right) \, ,
\]
where we bound $|\mathfrak{Z}|$ by $\mathcal{T}_M$. The function $x [\log (x) -1]$ is bounded below by $-1$ for any positive $x$; applying this bound to $V(\mathfrak{Z})[\log V(\mathfrak{Z}) -1]$ and rearranging the previous inequality algebraically, we find that, for all sufficiently large $n$
\[
V(\mathfrak{Z}) \leq \frac{C \xi +  \log \left(\frac{\mathcal{T}_V}{\tilde{\tau}_s}\right)+ 1}{\log[q \tilde{\tau}_s/ (w \cdot \mathcal{T}_M) ]}  \leq \frac{C \xi +  \log \left(\frac{\mathcal{T}_V}{\tilde{\tau}_s}\right) + 1}{(a/3) \cdot \log n} \leq \frac{C\left[1 + \log \left(\frac{\mathcal{T}_V}{\tilde{\tau}_s}\right)\right]}{\log n }\, ,
\]
where use \eqref{eq:PolyLowerBound} to get the penultimate bound. This completes the proof. 
\end{proof}

Recall that, for any $\mathfrak{W} \subset T$, $Q(\mathfrak{W}) \leq V(\mathfrak{W})^2$. In a sense, the content of Lemma \ref{lmm6} is that this bound is nearly right for $\mathfrak{T}$. To make this precise, we define 
\[
P_I(\mathfrak{W}) := \frac{1}{q} \sum_{J \in \mathfrak{W} ,\, \rho(I,J) >s} X_J \, .
\]
Note that the sum is over indices whose distance {\em exceeds} $s$.
\begin{coro}\label{PIbound}
Assume $\mathcal{H}_{\xi}$ holds, and that $\mathcal{T}_V$ and $\mathfrak{T}$ are defined as above. Then, for all $n$ sufficiently large and $\xi$ sufficiently small, 
\[
\frac{1}{q}  \sum_{I \in \mathfrak{T}} X_I P_I(\mathfrak{T}) \leq 3 \phi(\mathcal{T}_V)+ \psi(\mathcal{T}_V) \, ,
\]
where $\phi(\cdot)$ and $\psi(\cdot)$ are defined as in Lemma \ref{lmm6}. 
\end{coro}
\begin{proof}
We observe that
\[
V(\mathfrak{T})^2 - Q(\mathfrak{T}) = \frac{1}{q}  \sum_{I \in \mathfrak{T}} X_I P_I(\mathfrak{T})+ \frac{V(\mathfrak{T})}{q} \, , 
\]
where the additive factor of $V(\mathfrak{T})/q$ comes from the fact that $n^2 - 2 {n \choose2} = n$.
By Lemma~\ref{lmm6}, we conclude that
\begin{align*}
\frac{1}{q}  \sum_{I \in \mathfrak{T}} X_I P_I(\mathfrak{T}) &\leq V(\mathfrak{T})^2 - Q(\mathfrak{T}) \\
&\leq (1 + \phi(\mathcal{T}_V))^2 - (1- \psi(\mathcal{T}_V)) \leq 3 \phi(\mathcal{T}_V) + \psi(\mathcal{T}_V) \, .
\end{align*}
This completes the proof.
\end{proof}

\subsection{Removing Lower Order Terms} 
Before proceeding, consider the situation in which we assume that both $\phi$ and $\psi$ vanish, and that $\xi = 0$. This implies that $P_I(\mathcal{T}_V)$ must be zero for every $I \leq \mathcal{T}_V$, since the sum in Corollary \ref{PIbound} is made up of non-negative terms. Thus, $\mathfrak{T}$ would have diameter at most $s$. Thanks to the lower bound on $\mathcal{T}_V$ in Lemma \ref{lmm6}, the set would be a maximal clique set!

Of course, $\phi$, $\psi$, and $\xi$ are nonzero, so we cannot apply this argument to $\mathfrak{T}$ directly. We further truncate the set to deal with this difficulty. Define 
\begin{equation}\label{eq:mathcalTP}
\mathcal{T}_P := \max \left\{ k \leq \mathcal{T}_V : X_k \geq  \frac{\xi q}{\ttau} \right\} \, ,
\end{equation}
where we set $\mathcal{T}_P=0$ if the set on the right is empty. We denote the set $\{1, \dots, \mathcal{T}_P\}$ by $\mathfrak{P}$. The following lemma establishes bounds on $V(\mathfrak{P})$ and the sum of the $Y_I$'s in $\mathfrak{P}$. At the end of this section, we will use these bounds to deduce some geometric properties of $\mathfrak{P}$.

\begin{lem}\label{lmm7}
Assume that $\mathcal{H}_\xi$ holds, and define $\mathcal{T}_P$ and $\mathfrak{P}$ as above. Then, for sufficiently small $\xi$ and sufficiently large $n$,
\[
1 - \xi^{1/2} <  V(\mathfrak{P}) \leq 1 + \phi(\mathcal{T}_V)
\]
and
\[
\frac{1}{q} \sum_{I \in \mathfrak{P}} Y_I  < V(\mathfrak{P}) (\log(q/w) - 1) + \xi^{1/2}\, .
\]
\end{lem}
The stipulation on the sum of the $Y_I$'s in $\mathfrak{P}$ is a slight (but essential) improvement on the naive inclusion bound given by \eqref{bound2} in the definition of $\mathcal{H}_\xi$. 
\begin{proof}
The upper bound on $V(\mathfrak{P})$ follows from the inclusion $\mathfrak{P} \subset \mathfrak{T}$ and Lemma \ref{lmm6}. Define 
\[
\mathfrak{L}_1 := \biggl\{ I \in \mathfrak{T}: X_I < \frac{\xi q}{\ttau \log (\mathcal{T}_V)}\biggr\}
\]
and 
\[
\mathfrak{L}_2 := \biggl\{ I \in  \mathfrak{T}: \frac{ \xi q}{\ttau \log (\mathcal{T}_V)} \leq X_I <\frac{\xi q}{\ttau}
\biggr\}\, .
\]
Clearly, $\mathfrak{L}_1,\mathfrak{L}_2$ and $\mathfrak{P}$ form a partition of $\mathfrak{T}$. Proving the lemma is tantamount to 
proving good upper bounds on $V(\mathfrak{L}_1) $ and $V(\mathfrak{L}_2)$, as well as good lower bounds on the sum of the $Y_I$'s in both sets.

To bound $V(\mathfrak{L}_1)$, we first need to bound $P_I(\mathfrak{L}_1)$ from below for any $I \in \mathfrak{T}$. The worst case scenario is that the distance restriction removes
the $\NI$ largest elements of $\mathfrak{L}_1$. Therefore,
\[
P_I(\mathfrak{L}_1) \geq V(\mathfrak{L}_1) - \frac{1}{q}\NI \left( \max_{J \in \mathfrak{L}_1} X_J \right) > V(\mathfrak{L}_1) - \frac{\xi \NI}{\ttau \log (\mathcal{T}_V)}\, .
\]
Since $P_I(\mathfrak{W}) \leq P_I(\mathfrak{W}')$ whenever $\mathfrak{W} \subset \mathfrak{W}'$, we see that Corollary \ref{PIbound} implies that
\[
\frac{1}{q}  \sum_{I \in \mathfrak{T}}  X_I P_I(\mathfrak{L}_1) \leq 3 \phi(\mathcal{T}_V) + \psi(\mathcal{T}_V)\, .
\]
Replacing $P_I(\mathfrak{L}_1)$ with its minimum and recalling that $\NI$ and $\ttau$ are uniformly bounded in $n$ (by Lemma \ref{lmm3}), we see that
\[
\biggl(V(\mathfrak{L}_1) - \frac{C \xi}{\log (\mathcal{T}_V) } \biggr) V(\mathfrak{T}) < 3 \phi(\mathcal{T}_V) + \psi(\mathcal{T}_V)\, .
\]
Using the (very suboptimal) lower bound of $1/2$ for $V(\mathfrak{T})$ (which follows from Lemma \ref{lmm6} and $n$ sufficiently large), we conclude that
\begin{equation}\label{vl1}
 V(\mathfrak{L}_1) < 6\phi(\mathcal{T}_V) + 2 \psi(\mathcal{T}_V) + \frac{C \xi}{\log (\mathcal{T}_V)}\, .
\end{equation}
for some $C$ independent of $n$. Repeating this analysis with $\mathfrak{L}_2$ yields the inequality
\begin{equation*}
 V(\mathfrak{L}_2) < 6\phi(\mathcal{T}_V) + 2 \psi(\mathcal{T}_V) + C \xi\, .
\end{equation*}
Since both $\phi(x)$ and $\psi(x)$ are bounded above by $1/(\log n)^{1/2}$ (from \eqref{eq:UBpsiphi}) and $\mathcal{T}_V > (1 - \xi^{1/2})\ttau$ (by Lemma \ref{lmm6}), we get
\begin{equation}\label{vbounds}
\max\{V(\mathfrak{L}_1), \, V(\mathfrak{L}_2)\} < C \xi.
\end{equation}
Combining the previous bounds and the lower bound on $V(\mathfrak{T})$ from Lemma \ref{lmm6}, we find that, for all sufficiently small values of $\xi$, 
\begin{align}\label{eq:PartitionofV}
V(\mathfrak{P})&  = V(\mathfrak{T}) - V(\mathfrak{L}_1) - V(\mathfrak{L}_2) \\ &\geq  1 -  2 C \xi  - \frac{2 \xi}{\log n} \nonumber \\ &  \geq 1 - \xi^{1/2} \nonumber \, .
\end{align}
This establishes the lower bound on $V(\mathfrak{P})$.

We now turn to bounding $\sum_{I \in \mathfrak{P}} Y_I$. Suppose that
\begin{equation}\label{ineq32}
\frac{1}{q}\sum_{I \in \mathfrak{L}_i} Y_I > V(\mathfrak{L}_i)\left(\log(q/w) -1\right) - 2 \xi^{2/3}. 
\end{equation}
We observe that \eqref{eq:PartitionofV} and Lemma \ref{lmm6} imply that 
\[
V(\mathfrak{P}) + (2\xi)/\log n \geq 1 - V(\mathfrak{L}_1) - V(\mathfrak{L}_2).
\]
By inclusion, we may apply \eqref{bound2} to $\mathfrak{T}$, partition the elements into $\mathfrak{P}$, $\mathfrak{L}_1$, and $\mathfrak{L}_2$  and substitute the above bounds to get the following set of deductions 
\begin{align*}
\frac{1}{q}\sum_{I\in \mathfrak{P}}  Y_I & \leq (\log(q/w) - 1) + \xi - \frac{1}{q} \sum_{I \in \mathfrak{L}_1 \cup \mathfrak{L}_2} Y_I \\ & < \left(1 - V(\mathfrak{L}_1) - V(\mathfrak{L}_2) \right) \left(\log (q/w) - 1\right) + 4\xi^{2/3} + \xi
\\ & < V(\mathfrak{P})\left(\log(q/w) -1 \right) + \frac{2 \xi}{\log n} \left(\log(q/w) -1\right)  + 5\xi^{2/3} \\
& < V(\mathfrak{P})\left(\log(q/w) -1 \right) + \xi^{1/2}\, , 
\end{align*}
where the final inequality holds for sufficiently small $\xi$. Thus, we only need to show \eqref{ineq32}.

By applying Lemma \ref{lem:Jensen} to $\mathfrak{L}_i$, we know that
\begin{equation}\label{JensenLi}
\frac{1}{q}\sum_{I \in \mathfrak{L}_i } Y_I \geq V(\mathfrak{L}_i) \left(\log(q/w) -1 \right) +V(\mathfrak{L}_i) \log V(\mathfrak{L}_i) -
V(\mathfrak{L}_i) \log(|\mathfrak{L}_i|) \, ,
\end{equation}
where we ignored the positive term $V(\mathfrak{L}_i) \log \ttau$. Using \eqref{vbounds}, we find that $V(\mathfrak{L}_i) \log V(\mathfrak{L}_i)$ is bounded below by  $C \xi \log [C \xi] > - \xi^{2/3}$. 

Controlling $V(\mathfrak{L}_1) \log (|\mathfrak{L}_1|)$ is quite straightforward: since inclusion implies that $|\mathfrak{L}_1| \le  \mathcal{T}_V$, we can use \eqref{vl1} to see that, for sufficiently large $n$,
\begin{align*}
V(\mathfrak{L}_1) \log (|\mathfrak{L}_1|) & \leq \left(6 \phi(\mathcal{T}_V) + 2 \psi(\mathcal{T}_V) + \frac{C \xi }{\log(\mathcal{T}_V)} \right) \log(\mathcal{T}_V) \\ & \leq  \max_{2 \leq x \leq \mathcal{T}_M} \{ \log x \cdot (6 \phi(x) + 2 \psi(x)) \} + C \xi,
\end{align*}
where we use the fact that $2 \leq \ttau(1 - \xi^{1/2}) \leq \mathcal{T}_V \leq \mathcal{T}_M$. Recalling the definitions of $\phi$ and $\psi$, we have that
\begin{align*}
6 \phi(x) + 2 \psi(x) = \min &\left\{\frac{6C[\log(x/ \tilde{\tau}_s) + \xi]}{\log n}, \, \frac{12}{x}\right\}\,   \\ & \quad \, + \min \left\{ \frac{2C[1 + \log(x/ \tilde{\tau}_s)] }{\log n},\, \frac{2C'}{x} \right\} + \frac{\xi}{\log n},
\end{align*}
and thus, it can be shown that, for some (different) uniform constant $C$,
\[
\log x \cdot (6 \phi(x) + 2 \psi(x)) \leq C \cdot \min \left\{\frac{\log x (1 + \log x)}{\log n}, \frac{\log x}{x}  \right\} + \frac{\xi \log x}{\log n}.
\]
We wish to prove a uniform, $x$-independent bound on the minimum above (for any $x \geq 2$). The first function in the minimum is increasing; meanwhile, the second function in the minimum is larger than the first one on $[2,e]$ (for all sufficiently large $n$), and is a decreasing function of $x$ on $(e,\infty)$. Thus, for any $2 \leq x \leq \mathcal{T}_M$, the minimum is bounded above by $\log y(1+\log y)/\log n$ for any $y$ that satisfies $(1 + \log y)/\log n \geq 1/y$. The value $y = \log n$ is one such value. Combining the two estimates gives 
\[
\log x \cdot (6 \phi(x) + 2 \psi(x)) \leq \frac{C [\log \log n  + (\log \log n)^2]}{\log n} + \frac{\xi \log x}{\log n}.
\]
Thus, we find that 
\begin{equation}\label{eq:OptimizationOnTM}
\max_{2 \leq x \leq \mathcal{T}_M} \log(x) \cdot (6\phi(x)+ 2\psi(x)) \leq \frac{C [\log \log n  + (\log \log n)^2]}{\log n} + \frac{\xi \log \mathcal{T}_M}{\log n}. 
\end{equation}
Using \eqref{bound1} to bound $\mathcal{T}_M$,  we deduce that, for all sufficiently large $n$, $ V(\mathfrak{L}_1) \log (|\mathfrak{L}_1|) < C \xi$ for some uniform constant $C$; this, in turn, is bounded by $\xi^{2/3}$ for all sufficiently small $\xi$.

To control $V(\mathfrak{L}_2) \log (|\mathfrak{L}_2|)$, we must be slightly more careful. For any $I \in \mathfrak{L}_2$, we know that
\begin{align*}
P_I(\mathfrak{L}_2) & \geq \frac{1}{q} \sum_{J \in \mathfrak{L}_2, \rho(I,J) > s} \left( \min_{J \in \mathfrak{L}_2} X_J \right) \\ & \geq  \sum_{J \in \mathfrak{L}_2, \rho(I,J) > s} \frac{\xi}{\ttau \log (\mathcal{T}_V) }\\ & \geq \frac{\xi (|\mathfrak{L}_2| - \NI)}{\ttau \log (\mathcal{T}_V)},
\end{align*}
where we use the lower bound defining $\mathfrak{L}_2$. By inclusion, $P_I(\mathfrak{L}_2) < P_I(\mathfrak{T})$, and Corollary \ref{PIbound} allows us to conclude that
\begin{align*}
V(\mathfrak{T}) \left[\frac{\xi (|\mathfrak{L}_2| - \NI)}{\ttau \log (\mathcal{T}_V)} \right] & \leq \frac{1}{q}\sum_{I \in \mathfrak{T}} X_I
P_I(\mathfrak{L}_2) \\ & \leq 3 \phi(\mathcal{T}_V) + \psi(\mathcal{T}_V) \, .
\end{align*}
Solving for $|\mathfrak{L}_2|$, we deduce that
\[
|\mathfrak{L}_2| \leq \NI + \frac{\ttau}{\xi}\cdot \log (\mathcal{T}_V) \cdot  (6\phi(\mathcal{T}_V) + 2\psi(\mathcal{T}_V)) \, ,
\]
where we bound $V(\mathfrak{T})$ from below by $1/2$ by Lemma \ref{lmm6}. Referring back to \eqref{eq:OptimizationOnTM}, we see that
\begin{align*}
|\mathfrak{L}_2| \leq \NI + \frac{C \ttau \cdot [\log \log n  + (\log \log n)^2]}{\xi \log n} + \frac{\ttau \log \mathcal{T}_M}{\log n} \,.
\end{align*}
Using \eqref{bound1}  and Lemma \ref{lmm3} to bound $\NI$ and $\ttau$ from above, this proves that $|\mathfrak{L}_2|$ is uniformly bounded in $n$. Appealing to \eqref{vbounds} a final time, 
\[
V(\mathfrak{L}_2) \log (|\mathfrak{L}_2|) < C \xi < \xi^{2/3}.
\]
This completes the proof.
\end{proof}

As promised, we now show that $\mathfrak{P}$ has the desired geometric properties:
\begin{lem}\label{MaximalClique}
For all $n$ sufficiently large and $\xi$ sufficiently small, 
\[
\text{diam}(\mathfrak{P}) \leq s \quad \text{ and } \quad \mathcal{T}_P > \ttau \left(1 - \xi^{1/3}\right),
\]
that is - the set $\mathfrak{P}$ is a $\xi^{1/3}$-almost maximal clique set.
\end{lem}
\begin{proof}
Assume that there exists a pair of indices $I^*,J^* \in \mathfrak{P}$ such that $\rho(I^*,J^*) >s$. Then, 
\[
\frac{1}{q} \sum_{I\in \mathfrak{T}} X_I P_I(\mathfrak{T}) \geq \frac{X_{I^*} X_{J^*}}{q^2} \geq \frac{\xi^2}{\ttau^2},
\]
where the final bound is from the definition of $\mathfrak{P}$. By Corollary~\ref{PIbound}, we know the lefthand quantity cannot exceed $3 \phi(\mathcal{T}_V)+ \psi(\mathcal{T}_V)$, which is bounded above by $4/(\log n)^{1/2}$ - a clear contradiction for all sufficiently large $n$. Thus, the diameter $\mathfrak{P}$ is at most $s$.

Combining Lemmas \ref{lem:Jensen} and \ref{lmm7} gives
\begin{align*}
&V(\mathfrak{P}) \left(\log(q/w) - 1 \right) +  V(\mathfrak{P}) \left(\log V(\mathfrak{P}) - \log \left(\frac{\mathcal{T}_P}{\tilde{\tau}_s}\right)\right) \\
&< V(\mathfrak{P}) \left(\log(q/w) - 1 \right) + \xi^{1/2}\, ,
\end{align*}
and therefore
\[
\mathcal{T}_P > \tilde{\tau}_s \cdot V(\mathfrak{P}) \exp \biggl(- \frac{ \xi^{1/2}}{V(\mathfrak{P})}\biggr) >  \tilde{\tau}_s \cdot V(\mathfrak{P}) \biggl(1 - \frac{\xi^{1/2}}{V(\mathfrak{P})}\biggr)
\]
using the standard estimate $e^{-x} \ge 1-x$. Combining this with the lower bound on $V(\mathfrak{P})$ from Lemma \ref{lmm7} forces
\[
\mathcal{T}_P  > \tilde{\tau}_s (1 - \xi^{1/3})\, 
\]
for all sufficiently large values of $n$ and small values of $\xi$.
\end{proof}

\subsection{Convex Analysis}
We are nearly done with the proof: all that remains is to show that, for most $I \in \mathfrak{P}$, $X_I$ is close to $q/\tilde{\tau_s}$, and then to formally prove the theorem. The essential additional information we are now armed with is an {\em upper} bound on $\mathcal{T}_P$ --- namely $\tilde{\tau}_s$, as $\mathfrak{P}$ has diameter at most $s$, and $\tilde{\tau}_s$ is the largest possible cardinality for such a set of indices.

\begin{lem}\label{QuantitatveJensens}
Let $\mathfrak{P}$ be as above, and assume that $\mathcal{H}_\xi$ holds.  Define 
\[
\mathfrak{B} := \left\{ I: \left|\frac{X_I \tilde{\tau}_s}{q} - 1 \right| < \xi^{1/5} \right\} \quad \text{ and } \quad \mathfrak{C} = \mathfrak{P} \setminus \mathfrak{B}.
\]
Then, for all sufficiently large $n$ and $\xi$ sufficiently small, 
\[
|\mathfrak{B}| > (1 - 10 \xi^{1/10})\ttau, \quad \, |\mathfrak{C}| < 9 \xi^{1/10} \ttau, \quad \text{and} \quad V(\mathfrak{C}) < 10 \xi^{1/10}. 
\]
\end{lem}
\begin{proof}
We consider the Taylor expansion of $Y_I$ around the value $q/\tilde{\tau}_s$. Explicitly, we let $f(x) = x (\log(x/\mathcal{D}) -1) + \mathcal{D}$, and by Taylor's theorem,
\[
Y_I = f(X_I) = f\biggl(\frac{q}{\tilde{\tau}_s}\biggr) + f'\biggl(\frac{q}{\tilde{\tau}_s}\biggr) \biggl(X_I - \frac{q}{\tilde{\tau}_s}\biggr) +
\frac{f''(L(X_I))}{2} \biggl(X_I - \frac{q}{\tilde{\tau}_s}\biggr)^2
\]
where $L(X_I)$ is some number between $X_I$ and $q/\tilde{\tau}_s$. Differentiating $f(x)$ explicitly and simplifying algebraically, we see that
\[
Y_I = \mathcal{D} - \frac{q}{\tilde{\tau}_s} + X_I \log(q/w) + \frac{1}{2L(X_I)} \biggl(X_I - \frac{q}{\tilde{\tau}_s}\biggr)^2\, .
\]
Next, we sum over $\mathfrak{P}$ and use the upper bound on $\sum_{I \in \mathfrak{P}} Y_I$ from Lemma \ref{lmm7} to deduce that
\begin{align*}
 &\frac{1}{q}\displaystyle \sum_{I \in \mathfrak{P}}\biggl[\mathcal{D} - \frac{q}{\tilde{\tau}_s} + X_I\log(q/w) + \frac{1}{2L(X_I)} \biggl(X_I -
 \frac{q}{\tilde{\tau}_s}\biggr)^2 \biggr] \\
 &\leq V(\mathfrak{P}) (\log(q/w) -1) + \xi^{1/2}\, .
\end{align*}
We ignore the positive term $\mathcal{D}$ on the lefthand side. From Lemma \ref{MaximalClique}, the diameter of $\mathfrak{P}$ is at most $s$, and hence $\mathcal{T}_P \leq \ttau$. Thus,
\begin{equation}\label{lxi}
\frac{1}{q}\sum_{I \in \mathfrak{P}} \frac{1}{2L(X_I)} \biggl(X_I - \frac{q}{\tilde{\tau}_s}\biggr)^2 \leq \frac{\mathcal{T}_P}{\ttau} - V(\mathfrak{P}) + \xi^{1/2} \leq 2
\xi^{1/2} \, ,
\end{equation}
where the final inequality follows thanks to the lower bound on $V(\mathfrak{P})$ from Lemma~\ref{lmm7}.

Now, define
\[
\mathfrak{W}_1:= \left\{I \in \mathfrak{P}: X_I \geq (1 + \xi^{1/5}) q/\tilde{\tau}_s\right\}
\]
and
\[
\mathfrak{W}_2:=  \left\{I \in \mathfrak{P}: X_I \leq (1 - \xi^{1/5}) q/\tilde{\tau}_s \right\}\, .
\]
We recall that $\mathfrak{C} = \mathfrak{P} \setminus \mathfrak{B} =  \mathfrak{W}_1 \cup \mathfrak{W}_2$. On $\mathfrak{W}_1$, the function $1/L(X_I)$ is bounded below by $1/X_I$. Thus, \eqref{lxi} implies that
\[
 \frac{|\mathfrak{W}_1|}{q} \cdot \min_{I \in \mathfrak{W}_1} \left\{ \frac{1}{2X_I} \biggl(X_I - \frac{q}{\tilde{\tau}_s}\biggr)^2\right\} \leq 2 \xi^{1/2} \, .
\]
The function $x \mapsto (x - q/\tilde{\tau}_s)^2/(2x)$ is strictly increasing on the interval $[(q/\tilde{\tau}_s)(1 +\xi^{1/5}), \infty)$,  and always convex. Using the first fact, we find that
\begin{align*}
 |\mathfrak{W}_1| \cdot \left( \frac{\xi^{2/5}}{2\tilde{\tau}_s \cdot (1 + \xi^{1/5})} \right) \leq 2 \xi^{1/2}.
\end{align*}
This implies that $|\mathfrak{W}_1| < 5 \xi^{1/10} \tilde{\tau}_s$. To control $V(\mathfrak{W}_1)$, we apply (the standard version of) Jensen's inequality to the convex function $x \mapsto (x - q/\tilde{\tau}_s)^2/(2x)$. Algebraically manipulating the resulting expression gives 
\[
\frac{1}{2V(\mathfrak{W}_1)} \left(V(\mathfrak{W}_1) - \frac{|\mathfrak{W}_1|}{\ttau}\right)^2 \le \frac{1}{q}\sum_{I \in \mathfrak{W}_1} \frac{1}{2X_I} \biggl(X_I - \frac{q}{\tilde{\tau}_s}\biggr)^2 \leq 2 \xi^{1/2},
\]
where the final bound follows from \eqref{lxi}. Since the lefthand side is increasing in $V(\mathfrak{W}_1)$ whenever $ V(\mathfrak{W}_1) \geq |\mathfrak{W}_1| /\tilde{\tau}_s$, we can conclude that 
\[
V(\mathfrak{W}_1) \leq |\mathfrak{W}_1| /\tilde{\tau}_s + 2 \xi^{1/4} < 6 \xi^{1/10}.
\] 
For $\mathfrak{W}_2$, we can bound $1/L(X_I)$ from below by $\tilde{\tau}_s/q$. A final appeal to \eqref{lxi} gives that 
\[
 \frac{|\mathfrak{W}_2|}{q} \cdot \min_{I \in \mathfrak{W}_2} \left\{ \frac{\tilde{\tau}_s}{2q} \biggl(X_I - \frac{q}{\tilde{\tau}_s}\biggr)^2\right\} < 2 \xi^{1/2} \, .
\]
We can bound the minimum from below by assuming that some $X_I$ realizes the upper bound that defines $\mathfrak{W}_2$. In this case, it is immediate that $|\mathfrak{W}_2| \le  4 \xi^{1/10} \ttau$. Finally, $X_I \leq q /\ttau$ for each $I \in \mathfrak{W}_2$, and therefore, $V(\mathfrak{W}_2) < 4 \xi^{1/10}$. Putting these terms together, we find that 
\[
V(\mathfrak{C}) = V(\mathfrak{W}_1) + V(\mathfrak{W}_2) < 10 \xi^{1/10}.
\]
We also see that
\[
|\mathfrak{C}| = |\mathfrak{W}_1| + |\mathfrak{W}_2| < 9 \xi^{1/10} \ttau.
\]
The lower bound on $|\mathfrak{B}|$ follows from
\[
|\mathfrak{B}| = \mathcal{T}_P - |\mathfrak{C}|> \ttau(1 - \xi^{1/3}) -   9 \xi^{1/10} \ttau > \ttau\left(1 - 10 \xi^{1/10}\right),
\]
where the penultiamte inequality was proved in Lemma \ref{MaximalClique}.
\end{proof}

We have completed the proof of the difficult assertion in Theorem~\ref{thm2}; all that is left is to ensure the second stipulation holds.

\begin{proof}[Proof of Theorem~\ref{thm2}]
We recall the definition of $\mathcal{G}_{n,\tdelta}(\teps)$: there must a pair of sets $\mathfrak{B}$ and $\mathfrak{C}$ such that  

a)$\mathfrak{B}$ is $ \teps$-almost maximal clique set such that $I \in \mathfrak{B}$ implies
\[
\left|\frac{\tilde{\tau}_s X_I}{q} - 1 \right| < \tilde{\varepsilon},
\]

b) $\mathfrak{C}$ satisfies
\[
|\mathfrak{C}| < \teps \ttau \quad \text{ and } V(\mathfrak{C}) < \tilde{\varepsilon},
\]
and \\

\, \,c) whenever $J \in (\mathfrak{B} \cup \mathfrak{C})^c$,
\[
X_J < \frac{\teps \cdot q}{\ttau}.
\]

We set $\xi =(\teps/10)^{10}$. Then, whenever $\mathcal{H}_\xi$ holds, the sequence of assertions given by Lemmas \ref{lmm6}, \ref{lmm7}, \ref{MaximalClique} and \ref{QuantitatveJensens} assure us that, for $n$ sufficiently large and $\teps$ sufficiently small (interpreted according to our stated conventions), the sets $\mathfrak{B}$ and $\mathfrak{C}$ of Lemma \ref{QuantitatveJensens} satisfy the first and second conditions of $\mathcal{G}_{n,\tdelta}(\teps)$.

To show the final condition holds, we must show that all elements of $T \setminus \mathfrak{P}$ are small. If $\mathcal{T}_V > \mathcal{T}_P$, then $X_{\mathcal{T}_P +1} < \xi q/\ttau < \teps q/\ttau$. This implies the upper bound holds for every element outside of $\mathfrak{P}$, thanks to the ordering of the elements.

We are left with the scenario in which $\mathcal{T}_V = \mathcal{T}_P$. By definition (recall \eqref{eq:mathcalTV} and \eqref{eq:mathcalTP}), this means that $\mathfrak{P} = \mathfrak{T}$, and therefore $\mathcal{T}_V \leq \ttau$.  Formally, it is still possible that $X_{\mathcal{T}_V +1} \geq  \xi q/\ttau$. By Lemma \ref{lem:Jensen} and Lemma \ref{lmm6}, we deduce that  
\begin{align*}
\frac{1}{q} \displaystyle \sum_{I \in \mathfrak{T}}  Y_I &\geq V(\mathfrak{T}) \left(\log(q/w) +\log(V(\mathfrak{T})) - \log\left(\frac{\mathcal{T}_V}{\tilde{\tau}_s}\right )- 1 \right) \\
&\geq \left(1 - \frac{2 \xi}{ \log n}\right) \left(\log(q/w) - 1 - \frac{3 \xi}{\log n}\right)\, ,
\end{align*}
where the contribution of the term including the logarithm of $\mathcal{T}_V$ is non-negative and thus can be safely ignored. Note that, since we assumed $q > 3w$, the righthand side above is positive for all sufficiently large $n$. Therefore, \eqref{bound2} implies that
\begin{align*}
\frac{1}{q} \displaystyle \sum_{I \in \mathfrak{I}\setminus \mathfrak{T} } Y_I &\leq \left[\log(q/w) -1 + \xi \right] - \left(1 - \frac{2 \xi}{\log n}\right) \left[\log(q/w) - 1 - \frac{3 \xi}{\log n} \right] \\
& \leq \frac{ 4\xi \log(q/w)}{\log n} + \xi \\
& \leq C \xi\, .
\end{align*}
As before, the final inequality follows because $\log (q/w) \leq C \log n$ for some~$C$.

Suppose $X_{\mathcal{T}_V +1} \geq \teps q /\ttau = 10 \xi^{1/10} q/\ttau$. Then we find that
\begin{align*}
Y_{\mathcal{T}_V +1} & = X_{\mathcal{T}_V +1} \left(\log(X_{\mathcal{T}_V +1}/\mathcal{D}) - 1\right) +\mathcal{D} \\
&\geq  \xi q  \left(\log(q/w) + \log (10\xi^{1/10}) - 1 \right) + \mathcal{D} \, .
\end{align*}
 If we divide through by $q$, we find that this expression still grows with $n$, whereas the earlier upper bound is uniformly bounded. This is a contradiction, and we get the upper bound $X_{\mathcal{T}_V +1} < \teps q/\ttau$.

The above computation implies that any configuration in $\mathcal{H}_{(\teps/10)^{10}}$ is also in $\mathcal{G}_{n,\tdelta}(\tilde{\varepsilon})$. Taking complements and intersecting with $\mathcal{L}_n(\tdelta)$, we conclude that
\begin{align*}
&\mathbb{P}[\mathcal{G}_{n,\tdelta}(\tilde{\varepsilon})^c \cap \mathcal{L}_n(\tilde{\delta})] \\
&\leq \mathbb{P}[\mathcal{H}_{(\teps/10)^{10}}^c \cap \mathcal{L}_n(\tilde{\delta})] \leq 3 \exp( - q [\log(q/w) -1 + (\teps/10)^{10} /2] ),
\end{align*}
where the final inequality is Proposition \ref{prop:DiscreteStructureTheroem}.
\end{proof}
\section{Moving from Discrete to Continuous}\label{sec:Geometry}
This section has three parts: first, we prove several estimates that show the discrete setting of the $s$-graded model approximates the continuous geometry of $\mathbb{T}^d$. We then prove a proposition relating the random geometric graph structural event $\mathcal{F}_n(\varepsilon)$ with $\mathcal{G}_{n,\tdelta}(\teps)$, given by the structure theorem \ref{thm2} on the $s$-graded model, with appropriately chosen parameters. The second part is probabilistic, where we find a tight lower bound on the event that the number of edges in the random geometric graph exceeds its mean. We then assume Theorem \ref{thm2} and deduce Theorem~\ref{thm1} by choosing $\tdelta$ and $\teps$ judiciously and employing a correlation inequality. 

\subsection{Geometric Lemmas}
Recall that (as in \eqref{InnerOuterHullDef}), for any $K \subset \mathbb{T}^d$, we define $\mathfrak{R}(K)$ and $\mathfrak{O}(K)$ to be the maximal (resp. minimal) subsets of $T$ such that
\begin{equation}
\mathfrak{U}\left(\mathfrak{R}(K)\right) \subset K  \quad \text{and} \quad K\setminus K' \subset \mathfrak{U}\left(\mathfrak{O}(K)\right) \, ,
\end{equation}
where $K'$ is some subset of $K$ of Lebesgue measure $0$. Recall that $\lambda(\cdot)$ will be used to denote Lebesgue measure of sets. We begin by showing that this operation does not alter the measure of convex subsets $S$ of diameter at most $r$ by very much. Formally:

\begin{pro}\label{FatteningProp}
Fix $\varepsilon >0$, and let $S$ be a convex subset of diameter at most $2r$. Then, there exists an $s_0$ that depends only on $\varepsilon$, the dimension, and the choice of norm, such that whenever $s > s_0$, the inner hull $\mathfrak{R}(S)$ satisfies
\begin{equation*}
\lambda(S) -\frac{\varepsilon^2 \tau }{64} < \lambda\big(\mathfrak{U}(\mathfrak{R}(S)) \big) \leq \lambda(S)
\end{equation*}
Similarly, the corresponding inequality
\begin{equation*}
\lambda(S)  \leq \lambda \big(\mathfrak{U}(\mathfrak{O}(S)) \big)< \lambda(S) + \frac{\varepsilon^2  \tau }{64}
\end{equation*}
hold for the outer hull $\mathfrak{O}(S)$.
\end{pro}
Note that, for any fixed $S$, this result follows from the continuity of Lebesgue measure. The essential part of this proposition is that the choice of $s_0$ is uniform for all convex subsets $S$ that satisfy the assumptions of the proposition.

Instead of proving the proposition directly, we consider the following minor modification:
\begin{lem}\label{BoundaryFattening}
Let $S$ be a convex set as in Proposition \ref{FatteningProp}. Define
\[
(\partial S)_{l} := \left\{y : \| \partial S -y \| \leq l\right\}\, ,
\]
where $\|B-x\|$ is shorthand for $\inf_{b\in B} \|b-x\|$ for any set $B$. Then there exists an $s_0$, depending on $\varepsilon$, the dimension, and the norm, such that $s > s_0$ implies 
\[
\lambda\left( (\partial S)_{\varsigma/m} \right)< \frac{\varepsilon^2 \tau}{64} \, ,
\]
where $\varsigma$ is the diameter of a unit square (as in \eqref{DiameterDefinition}).
\end{lem}

This lemma implies Proposition~\ref{FatteningProp}, since
\[
S \subset \mathfrak{U}\left(\mathfrak{R}(S)\right) \cup (\partial S)_{\varsigma/m}
\]
and
\[
\mathfrak{U}(\mathfrak{O}(S)) \subset S \cup (\partial S)_{\varsigma/m} \, .
\]
Taking the Lebesgue measure of both sides and using subadditivity gives the two nontrivial bounds in Proposition \ref{FatteningProp}.
\begin{proof}[Proof of Lemma \ref{BoundaryFattening}]
Heuristically, the volume of $(\partial S)_{\varsigma/m}$ should be commensurate with the product of $\varsigma/m$ with the surface area of $S$. Since $S$ has diameter at most $2r$ and is a convex set (and therefore its boundary cannot be too convoluted), this surface area ought to be bounded above by $Cr^{d-1}$. To formalize this loose heuristic, we turn to the tools of geometric measure theory. Although this approach is standard in that field, we include a detailed proof for completeness.

Consider a function $f:\mathbb{R}^d \rightarrow \mathbb{R}$ that is Lipschitz with respect to the Euclidean distance, and a Borel set $A$. The Euclidean coarea formula \cite[pg. 248]{Federer} states that, with the functions as above,
\[
\displaystyle \int_{A} \|Df(x)\|_2 dx = \displaystyle \int_{-\infty}^{\infty}  H^{d-1}\left(A \cap f^{-1}(y) \right) dy\, ,
\]
where $\| \cdot \|_2$ is the Euclidean norm, and $H^{d-1}$ is the Hausdorff measure on the surface (effectively the surface area) and $Df$ is the gradient of $f$, which exists since  $f$ is almost everywhere differentiable.

The natural choice for our analysis is $f(x) =\|\partial S - x\|$. We begin by proving that $f$ is Lipshitz with respect to the Euclidean norm. We recall the
classical fact that all norms are equivalent in finite dimensional space, i.e. there exists two positive constants $c$ and $C$ such that, for all $x$ and $y$,
\[
c \|x -y\| \leq \|x -y\|_2 \leq C \|x- y\|\, .
\]
Pick two points $x$ and $y$, and let $a \in  \partial S$ be the point such that $f(x) =
\|x - a\|$ (this point exists because $\partial S$ is closed). Then
\[
f(y) - f(x) \leq  \|y - a\| - \|x - a\|   \leq \|x-y\| \leq C \|x-y\|_2\, .
\]
Similarly, $f(x)-f(y)\le C\|x-y\|_2$.

Thus, $f$ is differentiable almost everywhere. Pick an $x$ where the function is differentiable, and let $a$ be as before. Then, for any $t \in (0,1)$,
\[
f(x + t(a-x)) \le f(x) - t \|a-x\|\, ,
\]
by the properties of norms. Subtracting $f(x)$ from both sides, dividing by $t$ and letting $t\to 0$, we get
\[
\langle a-x, Df\rangle \le - \|a-x\|\ \, ,
\]
where $\langle \cdot, \cdot \rangle$ is the Euclidean inner product (or in this case, the directional derivative). Applying the Cauchy-Schwarz inequality, we conclude that
\[
\|a-x\| \leq \|a-x\|_2\|Df\|_2\, ,
\]
which implies, by the equivalence of  norms, that
\[
\|Df\|_2 \geq \frac{\|a-x\|}{\|a-x\|_2} \geq c\, .
\]
Letting $A = \{ x : f(x) \leq \varsigma/m\}$, we apply the Euclidean coarea formula to deduce that 
\begin{equation}\label{CoareaS}
 \lambda\left[(\partial S)_{\varsigma/m}\right] \leq C \displaystyle \int_0^{\varsigma/m} H^{d-1}\left(\{y : \|\partial S - y\| = z \} \right) dz   \, ,
\end{equation}
where $C$ is some (possibly different) universal constant. Note that, for sufficiently small $z$ and $T$ convex, the set $\{y :  \|\partial T -y\| = z \}$ has two connected components: one inside $T$ and the other outside of it. We wish to show that both are boundaries of convex sets. The outer one is the boundary of the affine sum of $S$ and the ball of radius $z$, which is known to be convex. The internal one is the boundary of 
\[
S^{(z)} := \{x : x \in S, \|\partial S - x\| \geq z\} \, .
\]
We claim that this set is also convex: it if it weren't, we could find $x,y \in S^{(z)}$ such that $w = t x + (1-t) y \not \in S^{(z)}$ for some $t \in (0,1)$. Let $v$ be the minimal length vector such that $w + v \in \partial S$. Then, since $\|v\| < z$ by definition, we can find an $\varepsilon$ sufficiently small such that  $w + (1+\varepsilon) v \not \in S$, while $x +(1 + \varepsilon)v$, $y +(1 + \varepsilon)v$ are both in $S$, contradicting convexity of $S$.

We complete the proof using Cauchy's surface area formula \cite[pg. 55]{Rota}: for any convex $S$, 
\[
H^{d-1}(\partial S) = C_d \int_{u \in \mathbb{S}^{d-1}} \lambda_{d-1}\left( S | u^{\perp} \right) du \,
\]
where $\mathbb{S}^{d-1}$ is the $d-1$-dimensional unit ball in $\mathbb{R}^d$, $\lambda_{d-1}$ is the $d-1$-dimensional Lebesgue measure, $S | u^{\perp}$ is the projection of $S$ onto the $d-1$ dimensional subspace perpendicular to $u$, and $C_d$ is a constant used to ensure the Hausdorff and Lebesgue measures are compatible.

We apply this formula to the two convex sets $(S)_z$ and $S^{(z)}$ (discussed above) whose boundaries give the two connected components of $\{y : \|\partial S - y\| = z \}$. The projection of either set onto a $d-1$-dimensional subset has diameter at most $2(r+z)$, and therefore, by the triangle inequality, is included in some ball of diameter $4(r+z)$. Thus, inclusion implies that
\[
\lambda_{d-1}\left((S)_z | u^{\perp} \right) \leq C_d (r+z)^{d-1} \quad \text{ and } \quad  \lambda_{d-1}\left(S^{(z)} | u^{\perp} \right) \leq C_d (r+z)^{d-1}
\]
for some constant $C_d$. Substituting this into \eqref{CoareaS} gives
\[
\lambda\left[(\partial S)_{\varsigma/m}\right] \leq C \int_{r}^{r+(\varsigma /m)} u^{d-1} du = C  [(r + \varsigma/m)^d - r^d] \, . 
\]
Increasing $s$, we can ensure that $[1 + \varsigma/(r m)]^d < 1 +  (2 d \varsigma)/(r m )$, and therefore,
\[
\lambda\left[(\partial S)_{\varsigma/m}\right] \leq \frac{C r^{d-1}}{m} \leq \frac{C \tau}{rm} < \frac{C \tau}{s -r} \, ,
\]
where the value of $C$ may change from one inequality to the next. Increasing $s$, we get the desired result.
\end{proof}

Next, we prove a lemma that proves that almost maximal clique sets are, essentially, discretized versions of balls of diameter $r$:
\begin{lem}\label{GeoLemma}
Fix $\eps >0$. Then there exists $s_0$ such that, for any $s > s_0$, $r$ sufficiently small, and $s^{-1/20}$-almost maximal clique set $\mathfrak{B}$ in the $s$-graded model, there exists a set of indices $\mathfrak{H} \subset T$ with $|\mathfrak{H}| < \eps \cdot  \ttau$ and $B$, a ball of diameter $r$ in $\mathbb{T}^d$, such that  
\[
B \subset \mathfrak{U}(\mathfrak{B} \cup \mathfrak{H} ) \quad \text{ and } \quad \mathfrak{U}(\mathfrak{B}) \subset (B)_{\eps r}.
\]
\end{lem}
The choice of the term $s^{-1/20}$ is somewhat arbitrary --- the above result will follow for any function which vanishes as $s$ grows. We framed the lemma as we did since the function $s^{-1/20}$ is used in the proof of Theorem \ref{thm1} below.

\begin{proof}
To prove this lemma, we go through the abstract framework of Hausdorff convergence of subsets of a metric space. Consider an abstract metric space $X$ imbued with metric $\iota$, and, for any $S \subset X$, define the $l$-fattening of $S$ as before, using the metric $\iota$ to measure distance. For any two $A,B \subset X$, the Hausdorff distance is defined as
\[
\iota_{H} (A,B) := \inf \{ l : A \subset (B)_l,\, B \subset (A)_l\}\, .
\]
If $X$ is compact in the topology defined by $\iota$, the space of closed subsets of $X$ makes a compact space with respect to this metric \cite[page 294]{Petersen}.

Recall that $T_s = \{1,2,\dots,m\}^d$ is a set of indices, where we add the subscript $s$ to emphasize the dependence on this variable. Let $\{\mathfrak{B}_{s}\}_{s \geq 2}$ be a sequence of $s^{-1/20}$-almost maximal clique sets, where $\mathfrak{B}_{s} \subset T_s$. Define $A$ to be some ball of radius $r(1 + 2\varsigma)$ in $\mathbb{T}^d$. For any set $S$ of {\em diameter} of at most  $r(1 + 2\varsigma)$, there exists a translate of $S$ which lies completely in $A$ -- this can be done by simply translating any point of $S$ to the center of $A$. Thus, there is a translate of $\mathfrak{U}(\mathfrak{B}_{s})$ that is a subset of $A$, since
\[
\text{diam}(\mathfrak{U}(\mathfrak{B}_{s})) \leq r + \frac{2 \varsigma}{m} \leq r \left(1 + \frac{2\varsigma}{s-r}\right) < r(1 + 2\varsigma),
\]
whenever $s$ is sufficiently large and $r$ sufficiently small.

Let $\tilde{A}$ and $\tilde{B}_{s}$ be the sets $A$ and the translate of $\mathfrak{U}(\mathfrak{B}_{s})$ that are inside $A$, scaled by $1/r$ and embedded in $\mathbb{R}^d$ (which is possible as long as $r$ is sufficiently small to not ``notice'' the torroidal geometry of $\mathbb{T}^d$). Thus, $\tilde{A}$ is a ball of radius $1 + 2 \varsigma$. By  Lemma \ref{lmm3}, we know that
\[
\text{diam}(\tilde{B}_{s}) \leq 1 + \frac{C}{s}, \quad \text{and} \quad \lambda(\tilde{B}_{s}) \geq \frac{\tau (1 - s^{-1/20})}{r^d} = \frac{\nu(1 - s^{-1/20})}{2^d},
\]
with the final inequality being the definition of $\tau$.

Since all the $\tilde{B}_{s}$ are closed sets embedded in a single compact metric space (namely, $\tilde{A}$), we can extract a subsequence $\tilde{B}_{s_k}$ which converges to some set $\tilde{B}$ in the Hausdorff metric. Passing to the limit, we see that $\tilde{B}$ must have diameter of at most $1$ and measure at least $\nu/2^d$. Quoting the isodiametric inequality again \cite[pg 94]{Burago},
\[
\lambda(\tilde{B}) \leq \nu \left(\frac{\text{diam} \tilde{B}}{2} \right)^{d}\,  .
\]
where equality holds if and only if  $\tilde{B}$ is a ball of diameter $1$. This implies that $\tilde{B}$ is, in fact, the ball of diameter $1$ (up to sets of measure zero). Indeed, if we take an arbitrary subsequence of $\{\tilde{B}_s\}$, we can extract a convergent sub-subsequence whose limit will have diameter $1$ and measure $\nu/2^d$ --- meaning any sub-subsequential limit is some ball of diameter $1$. Letting $\mathcal{B}$ be the set of all balls of diameter in $A$, we find that 
\[
\lim_{s \rightarrow \infty} \inf_{\tilde{B} \in \mathcal{B}}  \iota_{H} (\tilde{B}_s,\tilde{B}) = 0.
\]
Therefore, there is an $s_0$, such that, for any $s > s_0$, there exists some ball $\tilde{B} \in \mathcal{B}$ such that 
\[
\iota_H (\tilde{B}_{s},\tilde{B}) < \eps/(16d).
\]
Scaling by $r$, we find that, for any $s > s_0$, there is a ball $B$ of diameter $r$ in $\mathbb{T}^d$, such that
\[
B \subset \left( \mathfrak{U}(\mathfrak{B}_s) \right)_{\eps r/(16d)}  \quad \text{and} \quad \mathfrak{U}(\mathfrak{B}_s)\subset (B)_{\eps r/(16d)}.
\]
The second statement implies the required inclusion of $\mathfrak{U}(\mathfrak{B})$ in a $(B)_{\eps r}$. For the other direction, we note that $\left( \mathfrak{U}(\mathfrak{B}_s) \right)_{\eps r/(16d)} \subset (B)_{\eps r/8d}$. Set $\mathfrak{H} = \mathfrak{O}(B_{\eps r/8d}) \setminus \mathfrak{B}_s$. Since $(B)_{\eps r / 8d}$ is convex, we can use Proposition \ref{FatteningProp} and Lemma \ref{lmm3} to ensure that
\[
|\mathfrak{O}\left((B)_{\eps r/8d}\right)|  \leq m^d \cdot \tau\left[ \left(1 + \frac{\eps}{4d}\right)^d + \eps^2/64\right] < \ttau (1 + \eps/2) 
\]
where the final inequality holds for sufficiently large $s$ and sufficiently small $\varepsilon$. $\mathfrak{B}_{s}$ has cardinality at least $\tilde{\tau}_s (1 - s^{-1/20})$; if $s_0$ is sufficiently large to ensure $s^{-1/20} < \eps/2$, we find that $|\mathfrak{H} | < \eps \tilde{\tau}_s$, and inclusion guarantees that 
\[
B \subset \mathfrak{U}\left( \mathfrak{B}_{s} \cup \mathfrak{H} \right).
\]
This completes the proof.
\end{proof}

\subsection{Relating the $s$-Graded Model and Random Geometric Graph Structure Theorems}
Using the geometric information derived in the previous section, we wish to prove the following proposition:
\begin{pro}\label{GnFnProp}
There is some $\eps_0>0$ such that the following holds for any $\eps<\eps_0$. Take any $\delta >0$ and $\tdelta \in [(1 - \eps/16)\delta,\delta]$. Let $\mathcal{F}_n(\varepsilon)$ and $\mathcal{G}_{n,\tdelta}(s^{-1/20})$ be the events described in Theorems \ref{thm1} and \ref{thm2}, respectively. Then, there exist $n_0$ and $s_0$ depending only on $\eps$ and $\delta$, such that if $n\ge n_0$ and $s\ge s_0$, then  
\[
\mathcal{G}_{n,\tdelta}(s^{-1/20}) \subset \mathcal{F}_n(\varepsilon).
\]
\end{pro}
\begin{proof}
Set $\teps = s^{-1/20}$ and $s_0$ and $n_0$ to be the sufficiently large to ensure that Lemma \ref{lmm3} holds for any $s > s_0$ and $n \geq n_0$ (further conditions on $s$ may be imposed on later in the proof). Assume $\mathcal{G}_{n,\tdelta}(\tilde{\varepsilon})$ occurs, and let $\mathfrak{B}$ and $\mathfrak{C}$ be the sets described in Theorem \ref{thm2}. Then, by definition, 
\begin{equation}\label{ElementwiseBoundClique}
\left|\frac{ \tilde{\tau}_s X_I}{q} -1 \right| < \teps \quad \forall I \in \mathfrak{B},
\end{equation}
\begin{equation}\label{ExceptionalSetBound}
|\mathfrak{C}| < \teps \cdot \ttau \quad \text{and} \quad \sum_{I \in \mathfrak{C}} X_I < \tilde{\varepsilon} q,
\end{equation}
and 
\begin{equation}\label{ElementwiseBoundBulk}
X_I < \frac{ \teps q}{\ttau}  \quad \forall I \not \in \mathfrak{B} \cup \mathfrak{C} \, . 
\end{equation}

Since the above bounds are in terms of $q$, whereas $\mathcal{F}_n(\eps)$ is defined using $\sqrt{2 \delta \mu}$, we begin by bounding $\sqrt{2 \delta \mu}/q$ from above and below. $\tdelta \leq \delta$ by definition, and $\mu \geq \tmu (1 + C/s)^{-1}$ by Lemma \ref{lmm3}. Therefore, 
\[
\frac{\sqrt{2 \delta \mu}}{q} = \sqrt{\frac{\delta \mu}{\tdelta \tmu}}  \geq \left(1 + \frac{C}{s}\right)^{-1/2} \geq (1 -\teps),
\]
where the final inequality holds for all $s$ sufficiently large. Similarly, $\mu \leq \tmu$ and $\tdelta > \delta (1 - \eps/16)$ implies that 
\[
\frac{\sqrt{2 \delta \mu}}{q} = \sqrt{\frac{\delta \mu}{\tdelta \tmu}}  \leq \left(1 - \frac{\eps}{16}\right)^{-1/2} \leq 1 +\eps/8,
\]
Putting this together, we see that, for all sufficiently large $s$, \begin{equation}\label{eq:comparison}
(1 - \teps) \leq \frac{\sqrt{2 \delta \mu}}{q} \leq 1 + \eps/8.
\end{equation}

We now set $s$ to be sufficiently large so that, when we apply Lemma \ref{GeoLemma} to $\mathfrak{B}$ to produce $B$, a ball of radius $r$, and $\mathfrak{H} \subset T$ with  $B \subset \mathfrak{U}(\mathfrak{B} \cup \mathfrak{H})$, we can be certain that $|\mathfrak{H}| < (\eps/8) \cdot \ttau$ and that $\mathfrak{B} \subset (B)_{ (\nu \eps^2 r)/(d \cdot 2^{d+3})}$ (the slightly odd constants are chosen to make later computations simpler). We also require that $\teps < \eps^2/16$. Our goal is to show that $B$ will satisfy the conditions of $\mathcal{F}_n(\eps)$.

For any $S \subset \mathbb{T}^d$, it is straightforward to see that 
\[
\displaystyle \sum_{I \in \mathfrak{R}(S)} X_I \leq |\chi(S)| \leq \displaystyle \sum_{I \in \mathfrak{O}(S)} X_I\, .
\]
To prove the first condition of $\mathcal{F}_n(\eps)$, it is sufficient to show that, for any convex $S \subset B$, 
\[
\displaystyle \sum_{I \in \mathfrak{R}(S)} X_I > \left(\frac{\lambda(S)}{\tau} - \eps \right) \sqrt{2 \delta \mu} \quad \text{ and } \quad \sum_{I \in \mathfrak{O}(S)} X_I < \left(\frac{\lambda(S)}{\tau} + \eps \right) \sqrt{2 \delta \mu}.
\]
For the upper bound, we use Proposition \ref{FatteningProp} and Lemma \ref{lmm3} to see that 
\[
|\mathfrak{O}(S)| = m^d \lambda(\mathfrak{U}(\mathfrak{O}(S))) < \left(\frac{\lambda(S)}{\tau} + \frac{\eps^2}{64}\right)m^d \tau  \leq  \left(\frac{\lambda(S)}{\tau} + \frac{\eps^2}{64}\right) \ttau. 
\]
We also know that, for any $\mathfrak{W} \subset T$, 
\begin{equation}\label{eq:breakdown}
\sum_{I \in \mathfrak{W}} X_I \leq \sum_{I \in \mathfrak{C}} X_I + |\mathfrak{W}| \cdot \max_{I \not \in \mathfrak{C}} X_I < \left(\teps + \frac{|\mathfrak{W}| (1 + \teps)}{\ttau}\right) q,
\end{equation}
using \eqref{ElementwiseBoundClique}, \eqref{ExceptionalSetBound}, and \eqref{ElementwiseBoundBulk}. Applying this to $\mathfrak{O}(S)$ gives that 
\begin{align*}
\sum_{I \in \mathfrak{O}(S)} X_I & < \left[\teps + \left(\frac{\lambda(S)}{\tau} + \frac{\eps^2}{64}\right) (1 + \teps)\right] q \\ & \leq \left[\frac{\teps}{1 - \teps} + \left(\frac{\lambda(S)}{\tau} + \frac{\eps^2}{64}\right) \cdot \frac{1 + \teps}{1 - \teps}\right] \sqrt{2 \delta \mu},
\end{align*}
where the final inequality is \eqref{eq:comparison}. Since $\lambda(S)/\tau \leq 1$ (as $S \subset B$) and $\teps < \eps^2/16$, the righthand side is smaller than $[\lambda(S)/\tau + \eps]\sqrt{2 \delta \mu}$ for all $\eps$ sufficiently small, as required.

To get a lower bound on the sum of the $X_I$'s in $\mathfrak{R}(S)$, we observe that $S \subset B$ implies that $\mathfrak{R}(S) \subset \mathfrak{B} \cup \mathfrak{H}$, and therefore,
\begin{align*}
|\mathfrak{R}(S) \cap \mathfrak{B}| & \geq |\mathfrak{R}(S)| - |\mathfrak{H}| \\ & >  \left(\frac{\lambda(S)}{\tau} - \frac{\eps^2}{64}\right) m^d \tau  - \frac{\eps \cdot \ttau}{8}\\ & \geq \left(\frac{\lambda(S)}{(1 + \teps) \tau} - \frac{\eps^2}{64(1 + \teps)} - \frac{\eps}{8}\right) \ttau ,
\end{align*}
using Proposition \ref{FatteningProp} and Lemma \ref{lmm3}. By definition, $X_I \geq [q (1 - \teps)]/\ttau $ for any $I \in \mathfrak{B}$, and therefore 
\begin{align*}
\sum_{I \in \mathfrak{R}(S)} X_I & > \sum_{I \in \mathfrak{R}(S) \cap \mathfrak{B}} X_I \\ & > \frac{1 - \teps}{1 + \teps} \left[\frac{\lambda(S)}{\tau} - \frac{\eps^2}{64} - \frac{\eps}{4}\right] q \\ & > \frac{1 - \teps}{(1 + \teps)(1 + \eps/8)} \left[\frac{\lambda(S)}{\tau} - \frac{\eps^2}{64} - \frac{\eps}{4}\right] \sqrt{2\delta\mu}.
\end{align*}
Reusing the bounds $\lambda(S) \leq \tau$ and $\teps < \eps^2/16$ gives that, for all $\eps$ sufficiently small, the final lower bound is smaller than $[\lambda(S)/\tau - \eps]\sqrt{2 \delta \mu}$, as required.

Next, let $S'$ be a convex set disjoint from $B$ with diameter at most $r$ and $\lambda(S') > \eps \tau$. A slightly more involved version of \eqref{eq:breakdown} gives that
\begin{align*}
\sum_{I \in \mathfrak{O}(S')} X_I & < \sum_{I \in \mathfrak{C}} X_I + |\mathfrak{O}(S') \cap \mathfrak{B}| \cdot \max_{I \in \mathfrak{B}} X_I + |\mathfrak{O}(S') \cap (\mathfrak{B} \cup \mathfrak{C})^c| \cdot \max_{I \in (\mathfrak{B} \cup \mathfrak{C})^c} X_I \\ & \leq \left[\teps + \frac{|\mathfrak{O}(S') \cap \mathfrak{B}| \cdot (1 + \teps)}{\ttau} + \frac{|\mathfrak{O}(S') \cap (\mathfrak{B} \cup \mathfrak{C})^c| \cdot \teps}{\ttau}\right]q.  
 \end{align*}
Since $S' \cap B = \emptyset$, we know that $\mathfrak{U}(\mathfrak{R}(S')) \cap B = \emptyset$. We have assumed above that $\mathfrak{B} \subset (B)_{(\nu \eps^2 r)/(d \cdot 2^{d+3})}$, and therefore 
\[
\mathfrak{U}(\mathfrak{R}(S') \cap \mathfrak{B})  \subset (B)_{(\nu \eps^2 r)/(d \cdot 2^{d+3})} \setminus B.
\] 
Taking the measure of both sides and multiplying by $m^d$ to get cardinality bounds, we find that
\[
|\mathfrak{R}(S') \cap \mathfrak{B}| \leq m^d r^d \left[\left(1 + \frac{\nu \eps^2}{d 2^{d+3}}\right)^{d} -1 \right] < \frac{\eps^2 m^d \tau}{4} \leq \eps^2 \ttau/4. 
\]
Thus, 
\[
|\mathfrak{O}(S') \cap \mathfrak{B}| \leq |\mathfrak{O}(S') \setminus \mathfrak{R}(S')| + |\mathfrak{R}(S') \cap \mathfrak{B}| \leq (\eps^2/32 + \eps^2/4) \ttau < (\eps^2/2 )\ttau,
\]
using Proposition \ref{FatteningProp} to bound the difference between the inner and outer hulls of the convex set $S'$. Bounding $|\mathfrak{O}(S') \cap (\mathfrak{B} \cup \mathfrak{C})^c|$ by $|\mathfrak{O}(S')|$, using inequality \eqref{eq:comparison} to move from $q$ to $\sqrt{2 \delta \mu}$, and appealing to Proposition \ref{FatteningProp} one final time, we find that 
\begin{align*}
\sum_{I \in \mathfrak{O}(S')} X_I  & < \left[\teps + \frac{\eps^2 (1 + \teps)}{2} + \left(\frac{\lambda(S')}{\tau} + \frac{\eps^2}{32}\right)\teps \right] q \\ & \leq  \left[\frac{\teps \tau}{\lambda(S')} + \frac{2 \eps^2 \tau}{3 \lambda(S')} + \left(1 + \frac{\eps^2 \tau}{32 \lambda(S')}\right) \teps\right](1 - \teps)^{-1} \cdot \frac{\lambda(S') \cdot \sqrt{2 \delta \mu}}{\tau}. 
\end{align*}
Since $\lambda(S') > \eps \tau$ by assumption, and $\teps < \eps^2/16$, the product of bracketed term and $(1- \teps)^{-1}$ is bounded by $\eps$ for all sufficiently small values of $\eps$. This completes the proof.
\end{proof}
\subsection{Theorem \ref{thm2} implies Theorem \ref{thm1}}
The next lemma establishes a lower bound on the event $\{|E| > (1 + \delta) \mu\}$ --- the conditioning event in Theorem \ref{thm1}
\begin{lem}\label{LowerBoundH}
Fix $\delta >0$, and let $H_n(\delta) := \{|E| > (1 + \delta) \mu\}$. Then, setting $z = \max \{p/4, 3p/4 - 1/2\}$, there exists $n_0$ such that $n > n_0$ implies that
\[
\mathbb{P}[ H_n(\delta)] \geq \exp \left(-\sqrt{2 \delta \mu} \left[\log \left(\frac{\sqrt{2 \delta \mu}}{n \cdot \tau}\right) -1 \right] - C n^z \log n \right),
\]
where $C$ is an absolute constant independent of $n$ and $\delta$.
\end{lem}

\begin{proof}
Fix $B$, a ball of diameter $r$, and let $H'$ be the event that there are at least  $\lceil \sqrt{2 \delta \mu} + n^z \rceil$ vertices in $B$. Since the number of vertices in $B$ is a Poisson random variable of mean $n \tau$, we can get very straightforward lower bounds on the probability of $H'$:
\begin{align*}
\mathbb{P}[H'] & \geq  \mathbb{P}[\text{Poisson}(n \tau) =  \lceil \sqrt{2 \delta \mu} + n^z \rceil \, ] \\
&= \frac{\exp( - n \tau) \cdot (n \tau)^{\lceil \sqrt{2 \delta \mu} + n^z \rceil}}{(\lceil \sqrt{2 \delta \mu} + n^z \rceil)!}\, .
\end{align*}
By Stirling's approximation, it follows that, for sufficiently large $n$,
\[
\mathbb{P}[H'] \geq  \exp \left[ - n \tau  -  \lceil\sqrt{2 \delta \mu} + 2n^z\rceil \left(\log \left[\frac{\lceil \sqrt{2 \delta \mu} + n^z \rceil}{n \tau}\right] - 1 \right) - C \log n \right]\, ,
\]
for some universal constant $C$.  Thanks to our judicious choice of $z$, $n \tau << n^z \log n << \sqrt{2 \delta \mu}$; therefore, by increasing $n$ we can find a $C$ independent of $n$ that guarantees 
\begin{equation}\label{H'Bound}
\mathbb{P}[H'] \geq \exp \left(-\sqrt{2 \delta \mu} \left[\log \left(\frac{\sqrt{2 \delta \mu}}{n \cdot \tau}\right) -1 \right] - C n^z \log n \right).
\end{equation}
If $H'$ occurs, there exists a clique with at least $\delta \mu + n^z \sqrt{\delta \mu}$ edges in it (since $0 < z < p/2$, this quantity is a lower bound for the number of edges in a clique of $\lceil \sqrt{2 \delta \mu} + n^z \rceil$ vertices for all sufficiently large $n$). Conditional on $H'$, $H_n(\delta)$ occurs if the number of edges with at most one endpoint in $B$ exceeds $\mu - n^z \sqrt{\delta \mu}$. Let $|E|'$ be the number of edges with {\em no} endpoints in $B$. Letting $1_{i,j}$ be the indicator of an edge between vertices $i$ and $j$, we can see that
\begin{align*}
\mathbb{E}(|E|') & = \mathbb{E}\left[ {N \choose 2} \mathbb{E}(1_{1,2} \cdot 1_{v_1, v_2 \not \in B} \mid N) \right]\\
&= \frac{n^2}{2} \mathbb{P}(\{\|v_1-v_2\|\le r\} \cap \{v_1, v_2 \not \in B\}) \, ,
\end{align*}
where $N$ is the total number of point in the torus, as before, and the probability measure in the second equality is given by the uniform process. For notational convenience, let $1^B_{i,j}$ be the indicator of the event $\{\|v_i-v_j\|\le r\} \cap \{v_i, v_j \not \in B\}$, and $\mu^B$ be its expectation under the measure of the uniform process (by symmetry, this is independent of the indices $i$ and $j$). If $v_1$ is at a  distance greater than $r$ from $B$, the second condition holds trivially. For a fixed $B$, the probability that $v_1$ is within distance $r$ of $B$ is a constant multiple of $r^d$.
Thus,
\[
\mu^B  \geq (1 - Cr^d) \nu r^d\, ,
\]
for some $C$ that depends only on the norm and the dimension. Thus, the expected value of $|E|'$ is bounded below by $\mu (1 - C r^d)$. Therefore, by definition of $z$, we have the inequality
\begin{equation}\label{expbd}
\mu - \sqrt{ \delta \mu} \cdot n^z \leq \mathbb{E}[|E|'] - \frac{\sqrt{ \delta \mu} \cdot n^z}{4}
\end{equation}
We now need a variance estimate for $|E|'$:
\[
\mathrm{Var}(|E|') = \mathbb{E}(\mathrm{Var}(|E|' \mid N)) + \mathrm{Var}(\mathbb{E}(|E|' \mid N))\, .
\]
We have already calculated the expectation of $|E|'$ given $N$ above; since $\mu^B$ does not depend on $N$, we deduce that
\[
\mathrm{Var}(\mathbb{E}(|E|' \mid N)) = (\mu^B)^2 \mathrm{Var} \left[{N \choose 2}\right] \,.
\]
A standard calculation will show that the variance of ${N \choose 2}$ is $n^3 + n^2/2$. Meanwhile,
\[
\mu^B \leq \mathbb{P}(\|v_1-v_2\|\le r) = \nu r^d \,.
\]
Combining these facts gives
\[
\mathrm{Var}(\mathbb{E}(|E|' \mid N)) \leq C r^{2d} n^3 \,,
\]
for some universal constant $C$.

Next, we estimate the expression $\mathbb{E}(\mathrm{Var}(|E|' \mid N))$. We can write this variance as
\[
\mathrm{Var}(|E|' \mid N) = \mathbb{E}\left[\Big(\displaystyle \sum_{1\le i<j\le N} [1^B_{i,j} - \mu^B ] \Big)^2 \mid N\right]  \, .
\]
We now decompose this sum into three sums by distributing the square: one sum over pairs of the form $(i,j),(k,l)$ with four distinct indices, one with pairs of the form $(i,j),(i,k)$ where one index repeats, and the final over perfect squares of terms involving $(i,j)$. The expectation of the first one is zero, as the event that $i,j$ form an edge with both endpoints outside of $B$ is completely independent of the same event occurring over distinct vertices $k,l$. For a fixed choice of $(i,j)$ and $(i,k)$, we can bound
\[
\mathbb{E}[1^B_{i,j} \cdot 1^B_{i,k}] \leq \mathbb{P}[ \| v_i - v_j\| \le r, \| v_i - v_k\| \le r] = (\nu r^d)^2 \, ,
\]
where the first inequality follows by removing the requirement that the vertices lie outside of $B$, and thus increasing the probability. There are $N(N-1)(N-2)$ ways to choose a pair of indices that overlap in exactly one entry. Thus,
\[
\displaystyle \sum (1^B_{i,j} - \mu^B)(1^B_{i,k} - \mu^B) =  \displaystyle \sum \left(1^B_{i,j} \cdot 1^B_{i,k} - (\mu^B)^2\right) \leq C' r^{2d} N^3 \,
,
\]
for some universal constant $C'$. Again, this overestimates the value of this sum dramatically, but is sufficient for our purposes. Finally, the contribution of terms of the form $(1^B_{i,j} - \mu^B)^2$ to the sum is exactly ${N \choose 2} (\mu_B - \mu_B^2)$, which is bounded above by $C'' r^d N^2$. Combining these results, taking expectations over $N$, and then adding the contribution of the variance of the expectation from before, we conclude that
\begin{equation}
\mathrm{Var}(|E|') \leq C''' (r^d n^2 + r^{2d} n^3)\, ,
\end{equation}
for yet another universal constant $C'''$. $r^d n^2$ grows as $n^{p +o(1)}$, while $r^{2d} n^3$ grows as $n^{2p-1 + o(1)}$. Thus, the variance of $|E|'$ is $n^{p+o(1)}$ if $p \leq 1$, and $n^{2p-1 + o(1)}$ when $p > 1$.

By Chebyshev's inequality and \eqref{expbd},
\begin{align*}
\mathbb{P}[|E|' < \mu - \sqrt{\delta \mu} \cdot n^z] & \leq  \mathbb{P}\left[|E|' < \mathbb{E}[|E|'] - \frac{\sqrt{\delta \mu} \cdot n^z}{4} \right] \\
& \leq \frac{16\text{Var}(|E|')}{\delta \mu \cdot n^{2z}}\, .
\end{align*}
Regardless of the value of $p$, this quantity vanishes as $n^{-p/2 + o(1)}$, and therefore, with probability $1- \varepsilon$, $|E|'$ exceeds $\mu - \sqrt{\delta \mu} \cdot n^z$ for all sufficiently large~$n$.

To conclude, we note that
\begin{align*}
\mathbb{P}[H_n(\delta)] & \geq \mathbb{P}[H_n(\delta) | H'] \cdot \mathbb{P}[H'] \\ & \geq  \mathbb{P}[|E|' \geq \mu - \sqrt{2\delta \mu} \cdot n^z] \cdot \mathbb{P}[H'] \\ & \geq (1 - \varepsilon) \mathbb{P}[H']. 
\end{align*}
Substituting \eqref{H'Bound} completes the proof.
\end{proof}

\begin{proof}[Proof of Theorem \ref{thm1}]
Fix $ \delta,\varepsilon\in (0,1)$. In this proof, whenever we say ``$s$ sufficiently large'', we mean ``$s\ge s_0$ for some $s_0$ that depends only on $\delta$ and $\eps$, and the universal constant $\teps_0$ from Theorem \ref{thm2}'', and whenever we say ``$n$ sufficiently large'', we mean ``$n\ge n_0$ for some $n_0$ that depends only on $\eps$, $\delta$, and our choice of $s$''. 

Let $\tdelta_0= (1-\eps/16)\delta$ and $\tilde{\Delta}_0 = \delta$. Define $\tilde{\delta}$  to  satisfy
\[
\tilde{\delta} \tilde{\mu}_s=  \delta \mu(1 - 1/(\log n)^2).
\]
Although $\tilde{\delta}$ will depend on $n$, we have that $\tdelta_0 \le \tdelta \leq \tilde{\Delta}_0$ for $s$ and $n$ sufficiently large (where we use Lemma \ref{lmm3} to bound $\mu/\tmu$ from above and below). This allows us to apply Theorem \ref{thm2}. We also define  
\[
D_n := \left\{ |E_s| - |E| > (1+\tdelta) \tilde{\mu}_s - (1 +\delta) \mu \right\} = \left\{ |E_s| - |E| > \tilde{\mu}_s - \mu -   \frac{\delta \mu}{(\log n)^2} \right\}. 
\]
The definitions are chosen in such a way that $\{D_n \cap H_n(\delta)\}$ imply the event $\mathcal{L}_n(\tilde{\delta})$.

Let $\mathcal{F}_n(\varepsilon)$ be the event described in the statement of Theorem \ref{thm1} --- i.e.~the existence of a ball $B$ housing the giant clique.
Our goal is to prove that 
\[
\lim_{n\rightarrow \infty} \mathbb{P}[\mathcal{F}_n(\varepsilon)^c | H_n(\delta)] = 0.
\]
By Proposition \ref{GnFnProp}, $\mathcal{G}_{n,\tdelta}(s^{-1/20})$ implies the event $\mathcal{F}_n(\varepsilon)$ whenever $\tdelta \in [(1 -\eps/16)\delta, \delta]$, $n$ and $s$ are sufficiently large, and $\eps$ is sufficiently small. Then, taking complements, we find that
\begin{align*}
\mathbb{P}[\mathcal{F}_n(\varepsilon)^c | H_n(\delta)] &\leq \mathbb{P}[\mathcal{F}_n(\varepsilon)^c \cap D_n | H_n(\delta)] + \mathbb{P}[D_n^c | H_n(\delta)] \\ & \leq \frac{\mathbb{P}[\mathcal{G}_{n,\tdelta}(s^{-1/20})^c \cap D_n \cap H_n(\delta)]}{\mathbb{P}[H_n(\delta)]} + \mathbb{P}[D_n^c | H_n(\delta)] \\ & \leq  \frac{\mathbb{P}[\mathcal{G}_{n,\tdelta}(s^{-1/20})^c \cap \mathcal{L}_n(\tilde{\delta}) ]}{\mathbb{P}[H_n(\delta)]} + \mathbb{P}[D_n^c | H_n(\delta)]
\end{align*}
We begin with the second term. To analyze it, we will use an instance of the celebrated FKG correlation inequality. First stated in the context of finite lattice systems \cite{FKG}, we will use a version that first appears in \cite{JansonFKG}. There is a natural partial ordering on configurations of $\chi$: we say $\omega \prec \omega'$ if $\omega \subset \omega'$ --- i.e.~every point of $\omega$ is also in $\omega'$. An event $A$ is increasing with respect to this ordering if $\omega \prec \omega'$ and $\omega \in A$ implies $\omega' \in A$. Heuristically, $A$ is increasing adding points to the configuration only makes $A$ more likely to occur. 
\begin{lem}[The FKG Inequality \cite{JansonFKG}]\label{FKG}
Let $A$ and $B$ increasing events. Then 
\[
\mathbb{P}[A \, | \, B] \geq \mathbb{P}[A].
\]
\end{lem}
It is clear that $H_n(\delta)$ is an increasing event. In addition, we can write $|E_s| - |E|$ as a sum over pairs of points in the Poisson Point Process whose distance exceeds $r$, but whose corresponding indices  satisfy $\rho(I,J) \leq s$. Thus, $D_n$ is also an increasing event, and Lemma \ref{FKG} allows us to deduce that
\[
\mathbb{P}[D_n | H_n(\delta)] \geq \mathbb{P}[D_n]. 
\]
Taking complements, we find that 
\begin{align*}
\mathbb{P}[\mathcal{F}_n(\varepsilon)^c | H_n(\delta)] \leq  \frac{\mathbb{P}[\mathcal{G}_{n,\tdelta}(s^{-1/20})^c \cap \mathcal{L}_n(\tilde{\delta}) ]}{\mathbb{P}[H_n(\delta)]} + \mathbb{P}[D_n^c] 
\end{align*}
It's easy to see that $\{|E_s| > \tmu - \delta \mu/[2 (\log n)^2]\}$ and $ \{|E| < \mu + \delta \mu/[2 (\log n)^2] \}$ imply $D_n$, and therefore, the union bound tells us that
\[
\mathbb{P}[D_n^c] \leq \mathbb{P}\left[|E_s| \leq \tmu - \frac{\delta \mu}{2 \log^2 n}\right] + \mathbb{P}\left[|E| \geq \mu + \frac{\delta \mu}{2 \log^2 n}\right].
\]
Recalling \eqref{eq:VarianceE} and \eqref{eq:VarianceEs}, we can see that
\begin{align*}
\text{Var}[|E|] &= \mu(1 + 2\nu n r^d) =  \max \{ n^{p+o(1)}, n^{2p-1 +o(1)} \},\\
\text{Var}[|E_s|] & \leq 16 |\tilde{S}_I| \NI^2 m^d \cdot  \max \{\mathcal{D}^2, \mathcal{D}^3\}  =  \max \{ n^{p+o(1)}, n^{2p-1 +o(1)} \}, 
\end{align*}
where we use the fact that $n r^d = n^{p-1 + o(1)}$ (which follows from \eqref{pdef}) for the first line, and equations \eqref{eq:mDef}, \eqref{eq:Ddef}, and the bounds on $|\tilde{S}_I|$ and $\NI$ from Lemma \ref{lmm3} to get control of the variance of $|E_s|$. Thus, Chebyshev's inequality gives that 
\[
\mathbb{P}[D_n^c] \leq \frac{4 \log^4 n (\text{Var}[|E|] + \text{Var}[|E_s|])}{\delta^2 \mu^2} = n^{ \max \{-p , -1\} + o(1)}.
\] 
As $n$ grows, this vanishes for all admissible $p$.

To complete the proof, we must show that, for $s$ sufficiently large,
\[
\lim_{n \rightarrow \infty} \frac{\mathbb{P}[\mathcal{G}_{n,\tdelta}(s^{-1/20})^c \cap \mathcal{L}_n(\tilde{\delta}) ]}{\mathbb{P}[H_n(\delta)]} = 0
\]
as $n$ grows. We will now use Theorem \ref{thm2}, which holds since $\tdelta$ is bounded above and below by $\delta(1 - \eps/16)$ and $\delta$, respectively. Reusing the convention $q = (2 \tilde{\delta} \tilde{\mu}_s)^{1/2}$ and $w = \tilde{\tau}_s \cdot \mathcal{D}$, the quantitative bound of Theorem \ref{thm2} and Lemma \ref{LowerBoundH} gives us 
\begin{align*}
\frac{\mathbb{P}[\mathcal{G}_{n,\tdelta}(s^{-1/20})^c \cap \mathcal{L}_n(\tilde{\delta}) ]}{\mathbb{P}[H_n(\delta)]} & \leq \exp \left( - q \left[\log \left(\frac{q}{w}\right) -1 \right] - \frac{q}{2} \left(\frac{1}{10 \cdot s^{1/20}}\right)^{10} \right. \\ & \quad \left. +  \sqrt{2 \delta \mu} \left[\log \left(\frac{\sqrt{2 \delta \mu}}{n \cdot \tau}\right) -1 \right] + C n^z \log n \right).
\end{align*}
Here, the choice $\tdelta \tmu = \delta \mu (1 - 1 /(\log n)^2)$ becomes essential. Careful algebra will show that the first order terms --- i.e.~those of order $\sqrt{2 \delta \mu} \cdot \log n$ --- will cancel perfectly. In fact, 
\begin{align*}
\sqrt{2 \delta \mu} \left[\log \left(\frac{\sqrt{2 \delta \mu}}{n \cdot \tau}\right) -1 \right] - q \left[\log \left(\frac{q}{w}\right) -1 \right] & = q \log\left(\frac{w}{n \tau}\right) + q \log \left(\frac{\sqrt{2 \delta \mu}}{q}\right)  \\ & \quad + (\sqrt{2\delta \mu}-q) \log \left(\frac{\sqrt{2\delta \mu}}{e n \tau}\right).
\end{align*}
By the choice of $\tdelta$, $\sqrt{2\delta \mu}/q = (1 - 1/(\log n)^2)^{-1/2}$, and hence, the bound $\log (1-x) \geq -2x$ for all sufficiently small $x$ implies that
\[
\log \left(\frac{\sqrt{2 \delta \mu}}{q}\right)  = -\frac{1}{2} \log \left(1 - \frac{1}{(\log n)^2}\right) \leq  \frac{1}{ (\log n)^2},
\]
for all sufficiently large $n$. Similarly, $\sqrt{2 \delta \mu} - q \leq 2 q/ (\log n)^2$ for all sufficiently large $n$ and $s$. Since $\sqrt{2 \delta \mu}/e n \tau = n^{(2-p)/2 + o(1)}$ (which follows from \eqref{pdef}), we conclude that
\begin{align*}
\sqrt{2 \delta \mu} \left[\log \left(\frac{\sqrt{2 \delta \mu}}{n \cdot \tau}\right) -1 \right] - q \left[\log \left(\frac{q}{w}\right) -1 \right] & \leq q \left[\log\left(\frac{w}{n \tau}\right) + \frac{3 - p}{2\log n}\right].
\end{align*}
Recalling that $w / (n\tau) = \ttau/(m^d \tau)$, we are left with 
\begin{align}\label{eq:FKGbound}
\nonumber&\frac{\mathbb{P}[\mathcal{G}_{n,\tdelta}(s^{-1/20})^c \cap \mathcal{L}_n(\tilde{\delta}) ]}{\mathbb{P}[H_n(\delta)]} \leq \\ &\quad \quad \quad \quad \exp \left( q \left[\log \left(\frac{\ttau}{m^d \tau}\right) - \frac{1}{2 \cdot 10^{10} s^{1/2}} + \frac{3-p}{2\log n}\right] + C n^z \log n \right).
\end{align}
From Lemma \ref{lmm3}, we can see that $\ttau / (m^d \tau) \leq (1 + C/s)$, and thus, for sufficiently large values of $s$, 
\[
\log \left(\frac{\ttau }{m^d \tau}\right) - \frac{1}{2 \cdot 10^{10} s^{1/2}} \leq \frac{C}{s} - \frac{1}{2 \cdot 10^{10} s^{1/2}} \leq - \frac{1}{4 \cdot  10^{10} s^{1/2}}.
\]
For all sufficiently large $s$ and $n$, we now see that the bracketed term of \eqref{eq:FKGbound} will be negative. Since $z < p/2$, we now see that the exponent approaches negative infinity as $n$ grows, and we deduce that 
\[
\lim_{n \rightarrow \infty} \frac{\mathbb{P}[\mathcal{G}_{n,\tdelta}(s^{-1/20})^c \cap \mathcal{L}_n(\tilde{\delta}) ]}{\mathbb{P}[H_n(\delta)]} = 0,
\]
as required.
\end{proof}

\section{Proof of the Upper Tail Large Deviation Principle}\label{LDPSection}
We now prove Theorem \ref{thmLDP}, which claims that the function
\[
I(x) := \left(\frac{2 -p}{2}\right) \sqrt{2 x}
\]
is the upper tail rate function for the random variable $|E|$ with speed $s(n) = \sqrt{\mu}\log n$. Recall that we restrict our attention to subsets of the interval $(0, \infty)$, as our result only holds for events in which $|E|$ exceeds its expectation.

Instead of proving Theorem~\ref{thmLDP} directly, we will prove the following proposition instead:
\begin{pro}\label{SuperExp}
Recall that, for any $\delta >0$, 
\[
H_n(\delta) = \{|E| > (1 + \delta) \mu\} =  \left\{\frac{|E| - \mu}{\mu} > \delta \right\}\, .
\]
Then, for any $\delta >0$ fixed, 
\begin{equation}\label{LDPInt}
\displaystyle \lim_{n\rightarrow \infty} \frac{\log \mathbb{P}[H_n(\delta)]}{\sqrt{\mu}\log n}  = -I(\delta) \,.
\end{equation}
\end{pro}
The equivalence of this proposition to Theorem \ref{thmLDP} is standard, but its proof is straightforward and we include it for completeness.
\begin{proof}[Proof that Proposition \ref{SuperExp} implies Theorem \ref{thmLDP}]
Pick $F$ to be a closed subset of $(0,\infty)$, and let $a_F > 0$ be its leftmost endpoint. Since $I(x)$ is increasing, its infimum over $F$ occurs at $a_F$. Furthermore, $F \subset [a_F,\infty)$, and therefore, for any $\eps >0$ satisfying $(a_F - \eps) >0$, 
\[
\mathbb{P}\left[\frac{|E| - \mu}{\mu} \in F \right] \leq \mathbb{P}\left[\frac{|E|-\mu}{\mu} \in (a_F - \eps,\infty) \right] = \mathbb{P}[H_n(a_F - \varepsilon)]\, .  
\]
Taking the logarithm, dividing by $\sqrt{\mu} \log n$, applying \eqref{LDPInt} and letting $\varepsilon$ go to zero gives the upper bound for $F$.

Next, take $G$ open in $(0,\infty)$ and pick $b \in G$. For some $\eps >0$, we know that $(b - \varepsilon, b] \in G$. Therefore,
\[
\mathbb{P}\left[\frac{|E| - \mu}{\mu} \in G \right] \geq \mathbb{P}\left[\frac{|E|- \mu}{\mu} \in (b- \varepsilon,b] \right] = \mathbb{P}[H_n(b - \varepsilon)] -
\mathbb{P}[H_n(b)]\, .
\]
Applying \eqref{LDPInt} twice, we deduce that, for any $\teps >0$, there is an $n$ sufficiently large to ensure that
\begin{align*}
\mathbb{P}\left[\frac{|E| - \mu }{\mu} \in G \right] &\geq \exp \left(-(1  + \teps) \cdot I(b - \varepsilon) \cdot  \sqrt{\mu} \log n \right) \\
&\qquad - \exp \left(- (1 - \teps) \cdot I(b) \cdot \sqrt{\mu} \log n \right)\, .
\end{align*}
Picking $\teps$ sufficiently small (as a function of $\varepsilon$) ensures that the second term is smaller than half the first term. Taking logarithms, dividing by $\sqrt{\mu} \log n$, and taking $\eps$ to zero establishes the lower bound on the probability of $\{(|E| - \mu)/\mu \in G \}$, and establishes Theorem~\ref{thmLDP}.
\end{proof}

\begin{proof}[Proof of Proposition \ref{SuperExp}]
The statement that, for any $\varepsilon >0$, there exists $n$ sufficiently large to ensure that
\[
\frac{\log \mathbb{P}[H_n(\delta)]}{\sqrt{\mu}\log n}  > - (1 + \varepsilon) I(\delta)
\]
is a direct consequence of Lemma \ref{LowerBoundH}. Thus, it will be sufficient an upper bound on the probability of $H_n(\delta)$.

Fix $\varepsilon >0$.  For an arbitrary pair of events $A$ and $B$, assume that, conditional on $A$, the event $B$ occurs with probability at least $1 - \varepsilon$. This implies that
\[
 \mathbb{P}[A] \leq \left(\frac{1}{1- \varepsilon}\right)\mathbb{P}[B]\, .
\]
By Theorem \ref{thm1}, there exists a sufficient large $n$ such that conditioning on $H_n(\delta)$ implies that the random geometric graph has a clique of size at least $\sqrt{2 \delta \mu}(1 - \varepsilon)$ with probability at least $1 - \varepsilon$. This means that, for any $s$, the $s$-graded model includes a maximal clique set $\mathfrak{P} \subset T$ with at least as many vertices as the clique of the random geometric graph. Since every maximal clique set has $\tilde{\tau}_s$ indices, there can be at most $m^{d \tilde{\tau}_s}$ distinct maximal clique sets; this is an egregious overcount, but we have no need
for finer control. Thus, by the union bound, the probability that there exists a maximal clique set with $\sqrt{2\delta\mu}(1 - \eps)$ vertices is bounded above by $m^{d \tilde{\tau}_s}$ times the probability that a single one has the same property. The number of vertices in a maximal clique set is distributed as a Poisson random variable of mean $\tilde{\tau}_s \cdot \mathcal{D} = w$. Therefore, the chain of implication allows us to conclude that
\[
\mathbb{P}[H_n(\delta)] \leq \left(\frac{m^{d \tilde{\tau}_s}}{1- \varepsilon}\right) \mathbb{P}\left[\text{Poisson}(w) > \sqrt{2 \delta \mu} (1 - \varepsilon)\right]\, .
\]
Let $v := \sqrt{2\delta \mu} (1-\varepsilon)$. Applying the Chernoff bounds of Poisson random variables (see Lemma \ref{ChernoffBounds}) to the right-hand side above gives
\begin{align*}
 \mathbb{P}[H_n(\delta)] \leq & \left(\frac{m^{d \tilde{\tau}_s}}{1- \varepsilon}\right) \exp\left( - v \left[\log\left(\frac{ v}{w}\right) -1 \right] - w
 \right)\\
& \leq \exp \left( - (1 - 2 \varepsilon) \sqrt{2 \delta \mu} \log \left[\frac{\sqrt{\mu}}{w}\right]\right) \, ,
\end{align*}
where the second inequality follows for all sufficiently large $n$ by noting that all the missing terms vanish in comparison to $\sqrt{\mu}  \cdot \log n$,
and can therefore be absorbed at the cost of changing $\varepsilon$ to $2 \varepsilon$. By the definitions of $\mu$, $p$ and~$w$,
\[
\frac{\sqrt{\mu}}{w} = n^{(2-p)/2 + o(1)} \, .
\]
Therefore, for any $\eta >0$, there exists an $n$ sufficiently large to ensure that 
\[
\frac{1}{\sqrt{\mu} \log n} \log \mathbb{P}[H_n(\delta)] \leq - (1 - 2 \varepsilon) \left(\frac{2-p}{2} + \eta \right) \sqrt{2 \delta}\,.
\]
Since $\varepsilon$ and $\eta$ are arbitrary, we conclude the desired upper bound.
\end{proof}

\vskip.2in
\noindent{\bf Acknowledgment.} The authors thank Mathew Penrose for many helpful comments. They would also like to thank the anonymous referees for their thorough reports, which greatly improved this manuscript, and helped correct several errors that were pointed out in previous versions.

\bibliographystyle{plain}

\end{document}